\documentclass[10pt]{article}
\usepackage{amsfonts}
\usepackage{amsmath,amsthm,amscd,amssymb,mathrsfs,setspace}
\usepackage{latexsym,epsf,epsfig}
\usepackage[makeroom]{cancel}
\usepackage{color}
\usepackage[hmargin=3cm,vmargin=3.5cm]{geometry}
\usepackage{graphicx}
\usepackage{hyperref}

\newcommand{\ds}{\displaystyle}

\newcommand{\xb}{{\bf{x}}}

\newcommand{\cA}{{\mathcal{A}}}

\newcommand{\cH}{\mathcal{H}}

\newcommand{\cD}{\mathscr{D}}
\newcommand{\Dn}{\partial_{\nu}}

\theoremstyle{plain}
\newtheorem{theorem}{Theorem}[section]
\newtheorem{lemma}[theorem]{Lemma}

\newtheorem{assumption}{Assumption}
\newtheorem{definition}{Definition}
\theoremstyle{remark}
\newtheorem{remark}{Remark}[section]
\numberwithin{equation}{section} \numberwithin{theorem}{section}
\numberwithin{remark}{section} 

\begin{document}

\title{Qualitative Results on the Dynamics of a Berger Plate with Nonlinear Boundary Damping}
\author{
\begin{tabular}[t]{c@{\extracolsep{2em}}c@{\extracolsep{2em}}c}
Pelin G. Geredeli \footnote{ {Hacettepe University}, {Dept. of
Mathematics}, {06800 Beytepe}, {Ankara, Turkey};
\textit{pguven@hacettepe.edu.tr} } & Justin T. Webster \footnote{
{College of Charleston}, Dept. of Mathematics, {Charleston, SC};
\textit{websterj@cofc.edu} }\end{tabular} } \maketitle

\begin{abstract}
\noindent The dynamics of a (nonlinear) Berger plate in the absence
of rotational inertia are considered with inhomogenous boundary
conditions. In our analysis, we consider boundary damping in two scenarios: (i) free
plate boundary conditions, or (ii) hinged-type boundary
conditions. In either situation, the nonlinearity gives rise to
complicating boundary terms. 
In the case of free boundary conditions we show that well-posedness of
finite-energy solutions can be obtained via highly nonlinear boundary dissipation. Additionally, we show the existence of
a compact global attractor for the dynamics in the presence of
hinged-type boundary dissipation (assuming a geometric condition
on the entire boundary \cite{lagnese}). To obtain the existence of the attractor we explicitly
construct the absorbing set for the dynamics by employing energy
methods that: (i) exploit the structure of the Berger
nonlinearity, and (ii) utilize sharp trace results for the
Euler-Bernoulli plate in \cite{sharptrace}. 

We provide a parallel
commentary (from a
 mathematical point of view) to the discussion of modeling with Berger
versus von Karman nonlinearities: to wit, we describe the
derivation of each nonlinear dynamics and a discussion of the validity  of the Berger
approximation. We believe this discussion to be of broad value
across engineering and applied mathematics communities.


\end{abstract}

\section{Introduction}

We consider the well-posedness and long-time behavior of
solutions to a boundary damped Berger plate equation taken in the \emph{%
absence of rotational inertia}. Its derivation and relation to
the scalar \emph{von Karman} nonlinearity \cite{springer,ciarlet,
lagnese} are discussed in depth in Section \ref{modeling}. We
consider the isotropic plate to be thin, as is usual in
\emph{large deflection} theory \cite{lagnese}. Thus we consider a
two dimensional domain $\Omega \subset \mathbb{R}^2$ with smooth
boundary $\Gamma$. The partial differential equation (PDE) model
of interest is then:
\begin{equation}  \label{plate}
\begin{cases}
u_{tt}+\Delta^2 u +\left[\gamma-\Upsilon\cdot\left( \ds
\int_{\Omega}\nabla u \cdot \nabla u ~d\Omega\right)\right]\Delta
u = p(\mathbf{x}) & \text{ in } (0,T)\times \Omega
\\
BC(u) & \text{ on } \Gamma \\
u(0)=u_0;~~u_t(0)=u_1. &
\end{cases}%
\end{equation}
The term $p\in L^2(\Omega)$ represents static pressure on the top
surface of the plate. The physical parameter $\Upsilon$ is fixed and positive, while $\gamma$ corresponds to in-plane stretching
($\gamma<0$) or compression ($\gamma>0$).
 In this treatment (i) we consider
the case $\gamma\ge0$, as is by now standard in
treatments of the
Berger plate \cite{ch-0,ch-l}, and (ii) we normalize $\Upsilon=1$. The term $%
BC(u) $ in \eqref{plate} represents the plate's boundary
conditions.

Throughout the paper, we will consider nonlinear damping acting
through two distinct types of boundary conditions; the nature of
damping will be critical to our well-posedness and long-time
behavior results. 
\begin{itemize} \item Free-clamped boundary conditions with dissipation, denoted {(FCD)}:
\begin{equation}
\begin{cases}
\Delta u+(1-\mu )B_{1}u=0~~\mathrm{on}~~\Gamma _{1},\notag \\
\partial _{\nu }(\Delta u)+(1-\mu )B_{2}u-\mu _{1}u=|u_{t}|^6 u_t~~\mathrm{on}%
~~\Gamma _{1},\notag \\
u=\partial _{\nu }u=0~~\mathrm{on}~~\Gamma _{0},%
\end{cases}%
 \label{freediss}
\end{equation}%
where $\Gamma =\Gamma _{0}\sqcup \Gamma _{1}$ is a disjoint union,
 $\Gamma _{0}\neq \emptyset $, and the boundary operators $B_1$
and $B_2$ are given by \cite{springer,lagnese}:
\begin{equation*}
\begin{array}{c}
B_1u = 2 \nu_1\nu_2 u_{xy} - \nu_1^2 u_{yy} - \nu_2^2
u_{xx}=-\partial_{\tau\tau}u-\left(\text{div}~ 
\nu\right)\partial_{\nu} u\;,
\\
\\
B_2u = \partial_{\tau} \left[ \left( \nu_1^2 -\nu_2^2\right)
u_{xy} + \nu_1
\nu_2 \left( u_{yy} - u_{xx}\right)\right]\,=\partial_{\tau}\partial_{\nu}%
\partial_{\tau}u.%
\end{array}%
\end{equation*}
$\nu=(\nu_1, \nu_2)$ is the unit outer normal to $\Gamma$, $\tau=
(-\nu_2, \nu_1)$ is the choice of unit tangent vector along
$\Gamma$. The parameter $\mu_1$ is nonnegative\footnote{Since $\Gamma
_{0}\neq \emptyset $, we could take  $\mu _{1}=0$ \cite{springer},
however, we retain it  for the sake of generality.}. The constant
$0<\mu<1$ has the meaning of the Poisson modulus.

\item Hinged boundary conditions with dissipation, denoted {(HD)}:
\begin{equation}
u=0;~~\Delta u=-D(\partial _{\nu }u_{t})~\mathrm{on}~~\Gamma,
 \label{hingeddiss}
\end{equation}
where the damping function $D(\cdot)\in C^1(\mathbb{R})$ is
monotone increasing, and $D(0)=0.$
\end{itemize}

\begin{remark}
It is clear that, barring issues of elliptic regularity, one can
consider various combinations of the above boundary conditions.
(This is done, for instance, throughout \cite{springer}.) In the
analysis of (FCD), we consider a disjoint portion of the boundary
to be clamped, primarily to simplify aspects of our analysis which
are not the principle focus, and to avoid technical issues of
elliptic regularity. One can take the entire boundary to be free,
so long as $\mu_1>0$.
\end{remark}

\begin{remark}\label{g1}
Another configuration (which is of physical interest) involves the so-called hinged-clamped boundary conditions.
One can consider the damping acting via moments on a {\em portion} of the boundary $\Gamma_1$ (again, disjoint from $\Gamma_0$---where the boundary is clamped).
This configuration, denoted by (HCD), is given by:
\begin{equation}
\begin{cases}
u=0;~~\partial_{\nu}u=0~\text{on}~~\Gamma_0 \\
u=0;~~\Delta u=-D(\partial _{\nu }u_{t})~\mathrm{on}~~\Gamma_1
 \label{hingedclampeddiss}
 \end{cases}
\end{equation}
We remark further on the (HCD) configuration in Remark \ref{g2}
below.
\end{remark}

For the Berger and von Karman evolutions, taken with {\em
homogeneous boundary conditions} (of clamped or hinged-type), the
theory is rather complete. However, in the two inhomogeneous cases
considered here the results we obtain differ dramatically from
those available for other nonlinear plate equations (e.g., von
Karman, Kirchhoff); this occurs since  {\em the advanced  theory for these evolution equations does not
apply}. In fact, the inhomogeneous boundary conditions are what
differentiate present analysis of the Berger plate from those
available in other available literature (see Section
\ref{previous}).

We seek to determine the effect of boundary damping on the
dynamics, and seek asymptotic behavior results (when possible)
under ``standard" assumptions on the plate domain and damping
functions.  We also provide a discussion of the validity of the
Berger model under these boundary contributions, including a
review of the pertinent engineering and mathematical literature.
By providing a modern analysis of the Berger evolution with
dissipation acting on the boundary, we attempt to reconcile the
hurdles in PDE analysis with those described in the engineering
literature. Indeed, the engineering literature reviewed here
provides numerical contrast between the von Karman and Berger
models in the presence of each of the standard boundary
conditions.

\vspace{.3cm} \noindent\textbf{Notation}: We will make use of
standard Sobolev spaces $W^{s,p}(\Omega)$
and denote norms in $H^{s}(\Omega)$ $(p=2)$ as $%
||\cdot||_s $ and $||\cdot||_0=||\cdot||_{L^2(\Omega)} $. We will
use the notation $(\cdot,\cdot)_{\Omega}$ for inner-products in
$L^2(\Omega)$ and $\langle \cdot, \cdot \rangle_{\Gamma}$ for
those in $L^2(\Gamma)$ (or a subset of $\Gamma$, as indicated by
the subscript where necessary). For simplicity, norms and inner
products written without subscript
 are taken to be $L^2$ of the appropriate domain (e.g., $(\cdot,\cdot)$ on $\Omega$ and $\langle \cdot, \cdot \rangle$ on $\Gamma$). The space
$W_0^{s,p}(\Omega)$ denotes the completion of
$C_0^{\infty}(\Omega)$ in the $W^{s,p}(\Omega)$ norm. We will use
the ubiquitous constant $C>0$, and denote critical dependencies
when necessary.

\subsection{Modeling Considerations}
\emph{This portion of the paper is  expository. It discusses the
origins of the von Karman and Berger nonlinearities, and provides
an engineering perspective on the validity of each model.}
\vskip.2cm
\label{modeling} Let $u^i(x,y,z,t)$ denote the $i$th component of
the displacement vector of the middle surface of an elastic plate
occupying the region $\Omega \subset \mathbb{R}^2$. Let
$\epsilon_{ij}$ and $\sigma_{ij}$ denote the standard strain and
stress tensors, respectively\footnote{Here, and below, we utilize the
summation convention and identify the spatial variables $x=x_1$,
$y=x_2$, and $z=x_3$.}. We assume the plate is homogeneous and
isotropic (with associated stress-strain relation).

The so called \emph{full von Karman system} is based on the
employment of a
variation principle applied to the energy, and the assumption of \emph{%
finite elasticity} \cite[Ch.2, pp. 13--20]{lagnese}. This is a
modification of the Kirchhoff approach, which is itself based on
(i) a linear strain-displacement relation and (ii) an assumption
that the linear filaments of the plate remain perpendicular to the
central plane of the plate throughout deflection. (The second
hypothesis is then linearized to produce the Kirchhoff equation.)
The full von Karman system is arrived at by replacing (i) above
with the nonlinear strain-displacement relation
\begin{equation*}
\epsilon_{ij}=\dfrac{1}{2}\left(\dfrac{\partial u ^i}{\partial x_j}+\dfrac{%
\partial u^j}{\partial x_i}\right)+\dfrac{1}{2}\dfrac{\partial u^k}{\partial
x_i}\dfrac{\partial u^k}{\partial x_j}.
\end{equation*}

The resulting system (which has a nice vectorial
representation---see \cite{koch} and references therein) consists
of a nonlinear beam equation in the transverse displacement,
coupled (nonlinearly) to a system of elasticity for the in-plane
displacements; much work has been done on this system \cite{koch,
puel} (and references therein). To simplify the structure of these
equations, one may neglect the in-plane acceleration components of
the model \cite{springer,ciarlet}. This allows for the decoupling
of the in-plane and transverse dynamics. The resulting
nonlinearity is referred to as the {\em scalar von Karman}
nonlinearity:
$$f_V(u)=-[u,v(u)+F_0],$$ where $F_0$ is a given in-plane load. The von Karman bracket  is employed:
\begin{equation}\label{vkbracket}[u,v]=u_{xx}v_{yy}+u_{yy}v_{xx}-2u_{xy}v_{xy}\end{equation}
 for all functions $u,v$ sufficiently smooth \cite{springer}.
The Airy stress function, $v(u)$ solves
$$\Delta^2 v =-[u,u]~\text{ in }~\Omega;~~v=\partial_{\nu}v =0~\text{ on } \Gamma.$$
\begin{remark}
We note that in the studies of nonlinear plates often rotational
inertia in the filaments of the plate (represented by the term
~$\displaystyle -(h^2/12)\Delta u_{tt}$ ~on the LHS of the plate
equation \eqref{plate}, where $h$ is the thickness of the plate)
is neglected. When $h$ is small (the plate is thin), this term is
discardable \cite{ciarlet,lagnese}. Mathematically, however, this term is
highly non-trivial and has substantial bearing on well-posedness
and long-time behavior results.
\end{remark}

Even with the simplification to a scalar equation, the complex
structure of the scalar von Karman nonlinearity (making use of a
nonlinear elliptic solver) provides incentive
 to attempt further simplification. We now paraphrase a review in \cite{inconsistent}:
\vskip.1cm \noindent {\em In 1955, Berger proposed in
\cite{berger} a modified von Karman system. This simplification is
based upon the assumption that the so called second strain
invariant of the middle surface is negligible.  The resulting
equations are simpler than von Karman's and Berger compared his
results to known solutions, and found good agreement. This
approach was generally accepted, though the hypothesis lacked a
clear mechanical interpretation. Yet, in many problems solved via
the Berger method the edges of the plates were assumed to be
restrained from in-plane movements, and, numerically, the
exactness of the method is associated to such restrictions. For
freely movable edges the accuracy of the results furnished by the
Berger approach becomes questionable. Indeed, the resulting
[numerical] accuracy depends on the order of the differential
operators appearing in the boundary conditions.
 In this way, hinged and free boundary conditions provide  more opportunities [than clamped conditions] for the solutions to become inaccurate or singular.
} \vskip.1cm Another analysis in \cite{studyberger} attempts to
find a rational mechanical basis for the Berger method. However,
the authors also conclude that the Berger results may not be a
consistently valid approximations across types of boundary
conditions. In \cite{newapproach}, it is elaborated that Berger's
line of thought leads to questionable results for movable edge
conditions owing to the fact that the neglect of second strain
invariant with movable edges fails to imply freedom of rotation in
the middle plane where the elastic stress exists. The novelty of
\cite{newapproach} is that the authors proceed to provide an
\emph{additional} nonlinear term in the equations of motion to
correct this, which they report has good accuracy with known solutions {\em for
all types of boundary conditions}.

\subsection{Previous Results and Mathematical Motivation}

\label{previous} \hspace{0.5cm}We now provide a general overview
of previous mathematical work done on the system in \eqref{plate}.
 Note that \emph{none of the results
below focus on inhomogenous boundary conditions};  to our knowledge
no previous works have provided a PDE analysis of the Berger plate with
dissipation acting on the boundary---though such analyses are
prevalent for the more ``complex" von Karman evolution.

The classical work in \cite{edenmilani} establishes results on the
existence of
exponential attractors for a more general Berger-like equation in $\mathbb{R}%
^n$ (which subsumes the standard Berger model), taken with
homogeneous hinged boundary conditions and fully-supported
interior damping. This reference discusses early work done by Ball (see
references in \cite{edenmilani}) on
extensible beam equations. The methods utilized by the authors of 
\cite{edenmilani} rely on the development of modified energy
functionals which allow one to circumvent the need for an explicit
construction of a Lyapunov function.\footnote{This is a key issue
below: in many nonlinear plate/extensible beam problems the
construction of a (strict) Lyapunov function requires a ``unique
continuation" result, which may not be generally available
\cite{g-l-w,pelin1}.}

 In the more recent book \cite{ch-2}, the Berger evolution with
homogeneous hinged boundary conditions (along with more general
nonlinear terms) is considered with fully-supported interior,
linear damping. Utilizing a Lyapunov approach, the existence of
smooth attractors is shown, but only in the presence of
(regularizing) \emph{rotational inertia} terms.

 In \cite{daniellorena} the authors study a non-dissipative von Karman
model arising in the context of \emph{piston theory}
\cite{bolotin,springer} for nonlinear flow-plate interactions.
This work also considers the case where rotational inertia is
absent, and the primary difficulties arise via (i) free boundary
conditions, as well as (ii) non-dissipative (piston theory) terms
appearing \eqref{plate}. The analysis does not concern the Berger
evolution, yet the authors must also combat non-dissipative terms
in the setting of free boundary conditions. 

In \cite{ma}, a Berger {\em beam} is considered with free-clamped boundary conditions; in this work stabilization is accomplished via a nonlinear boundary feedback acting in the third order free boundary condition. We also point out that to accommodate the non-dissipative terms which arise due to the interaction of the geometric nonlinearity and the free boundary condition (as discussed herein), the author {\em incorporates the nonlinear boundary contribution into the boundary condition at the free end of the beam}. Whether this is physical is unclear, however it does present an interesting mathematical consideration which could be investigated for free-clamped Berger plates.

The analysis in \cite{ji} deals with the boundary stabilization of
a Kirchhoff (polynomial) nonlinear plate, taken {\em with
rotational inertia}, and {\em linear damping} via a hinged
boundary condition (analogous to what we take). Though the
analysis is somewhat simplified (due to rotational inertia), this
work provides a guide for handling the non-dissipative,
second-order trace terms in the case of hinged boundary conditions
below.

 Finally, in perhaps the most modern account, the expansive monograph \cite
{springer} addresses the well-posedness and long-time behavior of
the scalar von Karman equations. However, this text also contains
 preliminary material on linear plates, and an abstract
presentation of nonlinear plates. Owing to this, many of the
results on well-posedness and long-time behavior  are valid for
Berger's evolution \emph{in the case of clamped or hinged} type
boundary conditions. However, if any type of free boundary
condition is employed (homogenous or inhomogeneous), the analysis
in \cite{springer}, although vast, does not apply. Moreover, the
best results in \cite{springer} for the von Karman evolution (with
no rotational inertia terms) with hinged dissipation (HD)
concerns {\em local compact attractors}. (In this reference they actually address the more general (HCD) conditions---allowing for a clamped portion {\em and} hinged dissipation portion (Remark \ref{g1})---with a geometric condition in force on $\Gamma_0$; their results apply to (HD) with no geometric condition by simply taking $\Gamma_0 = \emptyset$.) To obtain attractors in a
global sense, additional assumptions must be made (fully-supported
interior damping), and the resulting approach is {\em indirect}.
For a more detailed discussion relating these results to those
here, see the discussion following the presentation of Theorem
\ref{gatt} below.

In this treatment we are analyzing two cases where the Berger
dynamics have a marked difference from the von Karman dynamics. In
the case of free-clamped (FCD) type boundary conditions, Berger's
nonlinearity produces an additional non-dissipative term in the
dynamics which can only be accommodated via a {\em highly
nonlinear damping} to obtain well-posedness of the model. As such,
Section \ref{freeBC} can be viewed as a quantifiable mathematical
manifestation of the difficulties described above in applying
Berger's approximation to the von Karman model in the presence of
a free boundary component. For our analysis, we assume Berger's
hypothesis is in force and we analyze \eqref{plate} directly from
the equations, yielding non-dissipative effects which enter the
analysis at the level of finite energy.

In the case of hinged-type  boundary conditions (HD), after
 showing the well-posedness of the model (similar to what is done for (FCD)),
 we are interested in the asymptotic-in-time analysis. The standard nonlinear
 dissipation acting via moments on the boundary
is sufficient to obtain ultimate compactness of the dynamics (in
the form of a global attractor) {\em without any additional
damping or non-standard assumptions}---though we do require a
somewhat classical geometric assumption on the domain (albeit for
a {\em new reason}: the interaction between (HD) and the
nonlinearity).  In fact, we explicitly construct the {\em
absorbing ball} using techniques from \cite{c-e-l}, and along with
modern techniques  in the theory of dynamical systems (related to
the asymptotic smoothness property \cite{springer,kh}, see Sec.
4.2), we obtain the existence of a compact global attractor.

\subsection{Functional Setting}\label{grr}
The natural energy for linear plate dynamics is given by
\begin{equation*}
E(t)=E(u,u_t)=\frac{1}{2}\big\{
||u_{t}(t)||^{2}+a\big(u(t),u(t)\big)\big\}.
\end{equation*}
Here, $a(u,v)$ represents the potential energy and given by the
bilinear form
\begin{equation}
a(u,v)=\widetilde{a}(u,v)+\mu _{1}\int_{\Gamma _{1}}uv,
\label{a-uw}
\end{equation}%
where
\begin{equation}
\widetilde{a}(u,v)\equiv \int_{\Omega }(\Delta u \Delta v
-(1-\mu)[u,v]). \label{a-tild}
\end{equation}
Again, the notation $[u,v]$ corresponds to the von Karman bracket
(as introduced earlier in the text), and for all functions
sufficiently smooth \cite{springer}:
\begin{equation*}
\int_{\Omega} [u,v] d\Omega =
-\int_{\Gamma}\left((B_1u)\partial_{\nu} v-(B_2u) v
\right)d\Gamma.
\end{equation*}

\noindent Typically $\mu < 1/2$, but we allow for $\mu \le 1$ and note that  $\mu=1$ corresponds to the limiting case for
the boundary conditions, and the bilinear form $\widetilde
a(\cdot,\cdot)$ collapses to the standard bilinear form
$(\Delta\cdot,\Delta \cdot)_{\Omega}$ in the case of (HD)
boundary conditions. 

The above energy $E(t)$ dictates the state
space
\begin{equation}\mathscr{H}\equiv \begin{cases} H_{\Gamma _{0}}^{2}(\Omega)\times L^2(\Omega )& \text{ for }~(FCD),\\[.2cm]
(H^2\cap H_0^1)(\Omega)\times L^2(\Omega) & \text{ for
}~(HD).\end{cases}\end{equation}

\noindent Throughout this treatment, we will use the following
nonlinear energies associated to (\ref{plate}):
\begin{equation*}
\mathscr{E}(t)=\mathscr{E}(u,u_t)=E(t)+\Pi (u(t)),\hskip1cm
\widehat{E}(t)= E(t)+\frac{1}{4}||\nabla u(t) ||^{4},
\end{equation*}
where ~$\Pi$~ represents the non-dissipative and nonlinear
portion of the energy
\begin{equation}
\Pi(u)=\dfrac{1}{4}\big(||\nabla u||^{4}-2\gamma||\nabla
u||^2-4\int_{\Omega }pu \big). \label{pi}
\end{equation}

\subsection{Main Results and Discussion}

\label{results}

In this treatment the main results concern: well-posedness under
(FCD) and (HD) boundary conditions, as well as long time behavior of the
dynamics corresponding to (\ref{plate}) under (HD) boundary
conditions.

 Firstly, consider equation (\ref{plate}) with (FCD). We now
give a novel well-posedness result (see Section
\ref{abplate} for formal definitions of solutions):

\begin{theorem}
\label{mainresult0} With reference to \eqref{plate}, taken with
(FCD) boundary conditions and with initial data $(u_{0},u_{1})\in
\mathscr H$, for any $T>0$ there exists a unique generalized
solution $u\in C(0,T;\mathscr H)$ satisfying the energy equality
\begin{equation}
\mathscr{E}(t)+\int\limits_{s}^{t}\int\limits_{\Gamma
_{1}}|u_{t}|^8 +{\int\limits_{s}^{t}\int\limits_{%
\Gamma _{1}}(\gamma -\left\Vert \nabla u\right\Vert
^{2})(\partial_{\nu}u)u_t} = \mathscr{E}(s) \label{e-ineq}
\end{equation}%
for every $t>s>0$. The dynamics depend continuously on the initial data.
This implies that the map $(u(0),u_{t}(0))\mapsto (u(t),u_{t}(t))$
defines a strongly continuous semiflow $S_{F}(t)$ on $\mathscr{H}$.
\end{theorem}

\begin{remark}
We emphasize that with only {\em linear boundary damping} the
global existence of the solutions is an open problem, as one
cannot obtain the necessary a priori bounds. To overcome this
difficulty we consider a nonlinear boundary dissipation which
has an extreme (yet minimal with respect to our analysis)
polynomial structure.
\end{remark}

Secondly, we analyze the Berger model with (HD) boundary
conditions. It should be noted that under (HD)  conditions (unlike the (FCD) conditions)
 the structure of equation (\ref{plate})
does not yield any non-dissipative terms in the  energy relation.
This makes the direct application of the abstract plate theory
more straightforward for obtaining a well-posedness result. Before
giving this, the following assumption is needed on the damping:

\begin{assumption}
\label{ass2} There exists positive constants ~$0<m<M<\infty$ such
that
\begin{equation*}
m \leq D^{\prime}(s) \leq M , ~~ |s| \geq 1
\end{equation*}
\end{assumption}

\begin{remark}
For the purposes of well-posedness alone, one could impose weaker assumptions
on the damping mechanism $D(s)$; e.g., $D(s)$ could exhibit
arbitrary polynomial growth. However, for our subsequent results
concerning long time behavior it is necessary that $D(s)$ satisfy the
stronger assumption above which is also utilized in \cite{c-e-l,
springer}.
\end{remark}

Now we give the well-posedness result of \eqref{plate} taken with
(HD) boundary conditions:

\begin{theorem}
\label{wellpHD} Let Assumption \ref{ass2} hold. With reference to (%
\ref{plate}) taken with (HD) boundary conditions and with initial
data $(u_{0},u_{1})\in \mathscr H$, for all $T>0$ there exists a
unique generalized solution $u\in C(0,T;\mathscr H)$ depending
continuously on the initial data. This implies that the map
$(u(0),u_{t}(0))\mapsto (u(t),u_{t}(t))$ defines a strongly continuous
semiflow $S_{H}(t)$ on $\mathscr{H}$. The following energy
equality holds:
\begin{equation}
\mathscr E (t)+\int\limits_{s}^{t}\int\limits_{\Gamma }D
(\partial _{\nu }u_{t})\left(\partial _{\nu }u_{t}
\right)=\mathscr E (s)\label{enh}.
\end{equation}%
\end{theorem}
\begin{remark}
Assumption \ref{ass2} provides linear bounds for the damping
function from below and above which ensures the validity of the
energy equality (for more detail, see \cite[Section
4.2]{springer}).
\end{remark}
 Existence and properties of compact attractors for nonlinear plate models (e.g., von Karman,
Berger, Kirchhoff) have been studied in various contexts over the
last 30 years.\footnote{See Section \ref{previous}, also see the
Appendix for a discussion of dynamical systems and concepts
relevant to the treatment here.} In the present case, the
dynamical system associated to the Berger dynamics with (FCD)
cannot easily be shown to be dissipative. The absence of a strict
Lyapunov function (not readily available)  seems to preclude
showing the existence of a compact global attractor for the
dynamics. However, we can show {\em  dissipativity} of the
dynamical system under (HD) boundary conditions. Thus, by showing
the {\em asymptotic smoothness} property of the associated
dynamical system, the plate dynamics under (HD) will admit  the
existence of a \emph{compact global attractor} for finite energy,
generalized solutions. For this, we need an additional, standard,
star-shaped geometric assumption on
 the boundary \cite{springer,lagnese}:

\begin{assumption}\label{geom} There exists a point $\mathbf{x}_{0}\in \mathbb{R}^2$ such that $%
h(\mathbf{x})=\mathbf{x}-\mathbf{x}_{0}$ has the property that
$h\cdot \nu \ge 0 $ for all $\mathbf{x}\in \Gamma$, where $\nu $
is the outward normal vector to $\Gamma$.
\end{assumption}
\begin{remark}\label{star}
The star-shaped condition given above arises in boundary control,
when the boundary (or just a portion) is subjected to a given
feedback. In using state-of-the-art trace estimates
\cite{sharptrace}, we can accommodate boundary terms
associated with the linear dynamics {\em without a geometric
condition}. However, to deal with the nonlinear Berger
contribution on the boundary (in utilizing the flux multiplier
$h\cdot\nabla u$), we utilize
 the star-shaped assumption on $\Gamma$. Throughout \cite{lagnese}
 a star-shaped and star-complemented geometric condition (what is in Assumption
\ref{geom}, with the additional assumption that $h(\mathbf{x})
\cdot \nu \le 0$ on $\Gamma_0$) is in force for obtaining uniform
decay rates. Such a star-shaped, star-complemented assumption is
also used for various configurations in the context of von Karman
equations in \cite{springer}. We note that our assumption on $\Gamma$ is
precisely what \cite{sharptrace} can eliminate for linear
dynamics while still obtaining the same results of
\cite{lagnese}; however, at present it seems indispensable to
have an assumption on the boundary $\Gamma$ {\em due to the
interaction of the nonlinearity and the boundary conditions}.
\end{remark}

We note that at present the best result obtained for non-rotational von
Karman dynamics in the (HD) configuration (with additional damping and assumptions) is that of {\em local
attractors} (\cite[Theorem 10.5.7]{springer}). This is to say that
when one restricts the semiflow $(S_H(\cdot),\mathscr H)$ to a
ball of radius $R$ of initial data, a compact attractor can be
shown for the restricted dynamical system. This attractor is not necessarily uniform with respect to
large $R$. Such a result follows (though not immediately) from the
so called {\em asymptotic smoothness} property of the dynamics,
and estimates which yield a {\em local absorbing ball}. (See
\cite[Section 10.5]{springer} for this presentation.)  To obtain
a {\em global attractor} the authors of \cite{springer} utilize
the fact that a gradient dynamical system, which is also
asymptotically smooth, has a compact global attractor (Theorem
\ref{gradsmooth} in the Appendix here). To show the gradient property
of the dynamics, they rely upon (additional) fully-supported
interior damping. This damping yields a ``unique continuation"
property for the dynamics (or, from another point of view, allows for ``uniqueness"
in a compactness-uniqueness argument). This approach, and the requisite additional damping, is
 necessary because ``direct multipliers" do
not yield a uniform absorbing ball for the von Karman
dynamics---owing to boundary contributions and the structure of the ~$\Pi$~ for von Karman dynamics.

In the case of Berger's nonlinearity, we are able to adapt an
approach in \cite{c-e-l} to show the existence of an absorbing set
for the entire dynamics. {\em This is precisely because of the structure of the Berger nonlinearity}. The
discussion above brings us to the primary result concerning
long-time behavior of the Berger dynamics with (HD) boundary
conditions:
\begin{theorem}\label{gatt}
Consider (\ref{plate}) taken with (HD) boundary conditions; under
Assumptions \ref{ass2} and \ref{geom}, the dynamical system
$(S_H(t),\mathscr H)$ has a compact global attractor.
\end{theorem}

\begin{remark}\label{g2}
The above holds in the case of (HCD) boundary conditions (see Remark \ref{g1}), where $\Gamma=\Gamma_0 \sqcup \Gamma_1$ and the dissipation is active only on $\Gamma_1$. This will be addressed in the forthcoming manuscript \cite{new1}, which will also investigate further properties (such as dimensionality and smoothness) of the attractor given in Theorem \ref{gatt}. The key point in analyzing the partially-damped boundary configuration is to address the higher-order trace term $\Delta u$ on the uncontrolled portion of the boundary $\Gamma_0$. A geometric condition on the entire boundary (for instance, the star-complemented, star-shaped condition mentioned in Remark \ref{star}) is not sufficient in obtaining estimates which control the size of the absorbing ball. Rather, precise control of ~$\Delta u |_{\Gamma_0}$~ must be obtained via techniques which can be found, for instance, in \cite{george1,george2} (and references therein).
\end{remark}

\subsection{Technical Challenges and Contributions}\label{tech}

With regard to much of the multiplier analysis,
Berger's nonlinearity behaves more nicely than von Karman. And, in
the case of homogeneous hinged or clamped boundary conditions, the
analysis of the von Karman plate essentially {\em subsumes} the
analysis of the Berger evolution. However, under (FCD) boundary
conditions (with $F_0 \in H_0^2(\Omega)$), we have the disparity: \begin{align}
(f_V(u),u_t)_{\Omega}=(-[u,v(u)+F_0],u_t)_{\Omega} = &
~\dfrac{1}{2}\dfrac{d}{dt}\left\{\frac{1}{2}||\Delta
v(u)||_{\Omega}^2-([u,F_0],u)_{\Omega}\right\}\\[.2cm]
(f_B(u),u_t)_{\Omega}=\left((\gamma-||\nabla u||_{\Omega}^2)\Delta u ,u_t\right)_{\Omega}=&~ \dfrac{1}{2}\dfrac{d}{dt}%
\left\{\frac{1}{2}||\nabla u||_{\Omega}^4-\gamma ||\nabla
u||_{\Omega}^2\right\} \nonumber \\ + & \boxed{(\gamma-||\nabla
u||^2)\langle\Dn u, u_t\rangle_{\Gamma_1}}.\label{bergfreeterm}
\end{align}
The boxed term highlights the fundamental difference of the
dynamics, as the Berger-free dynamics contain an inherently
non-dissipative term with which we must contend.

In the case of (HD) boundary conditions, we note that the terms
of interest arise from the biharmonic term---which plays a key
role  in the multiplier analysis. We have for any $\phi $
sufficiently smooth
\begin{equation}
(\Delta ^{2}u,\phi )_{\Omega }=(\Delta u,\Delta \phi )_{\Omega
}+\langle\partial _{\nu }(\Delta u),\phi \rangle_{\Gamma}+\langle
D(\partial _{\nu }u_{t}),\partial _{\nu }\phi \rangle_{\Gamma}
\label{hingedbiharmonic}
\end{equation}%
In this case of $\phi = h \cdot \nabla u$ (for $h$ a ``nice" vector field, $u \in H^2(\Omega)$) the {\em damping on the boundary is pitted against a
high order boundary term}, which must be accommodated for any
sort of long-time behavior analysis.  Such trace terms must be
handled delicately. Here, we
are able to (via the special structure of the Berger nonlinearity)
control these boundary terms {\em directly}, in a way that allows
techniques from \cite{c-e-l} to obtain.

\subsubsection{Principal Challenges}
 The non-dissipative term above for (FCD) in \eqref{bergfreeterm} appears in the
 energy relation for the dynamics (and it is absent for von Karman dynamics).
 We follow the modern approach in \cite{springer},  using the abstract plate theory developed therein, and we must utilize
fully nonlinear damping of high polynomial degree.
 As such, we must address the technical difficulties associated to such damping
when utilizing energy methods. For obtaining dissipativity
and asymptotic smoothness of the dynamical system associated to
the (HD) dynamics, we  need to handle the nonlinear damping carefully.

In both configurations---(FCD) and (HD)---obtaining the results
in Section \ref{results} requires addressing the criticality (with
respect to the finite energy topology) of the Berger nonlinearity.
Similar to the non-rotational von Karman dynamics, the
Berger nonlinearity is not {\em compact}, and this
requires some care in energy estimates.

 Finally, what is perhaps the most challenging aspect of the analysis herein, is the
 difficulty involved in handling the trace terms associated to (HD) dynamics in the
 long-time behavior analysis. These terms include the standard linear terms---due to
 the principal biharmonic part---as well as ``new" boundary terms arising due to the interaction of the
 boundary conditions, requisite multipliers, and the Berger nonlinearity. We must deal carefully with signed
 terms on the boundary (including imposing a geometric assumption on the domain), as well as utilize
 the sharp trace results on second order traces (as in \cite{sharptrace}).

\subsubsection{Principal Contributions}
In what follows we address each of the points outlined above, as
well as make additional contributions:
\begin{itemize}
\item We show that a high (polynomial) degree of nonlinear damping can
accommodate
 (at least in terms of well-posedness) the free-type boundary condition (FCD).

\item We demonstrate asymptotic smoothness for the dynamical system
associated with
 the non-rotational Berger dynamics under (HD) conditions. This utilizes a 
 decomposition of Berger's nonlinearity, akin to the one utilized in the analogous
 result for non-rotational von Karman dynamics \cite[pp. 496--497]{springer}.

\item We critically use what is in \cite{ji} in considering (HD) conditions,
but we have a different treatment of trace terms
(particularly in the use of the sharp estimates in
\cite{sharptrace}). This has the benefit of extending the approach
in \cite{ji} to the non-rotational plates considered here.

 \item For the (HD) case, we are able to adapt the approach in
\cite{c-e-l} (corresponding to a semilinear wave equation with
boundary damping) to construct the compact global
attractor by directly obtaining estimates on the absorbing ball. This
depends in a critical way on the special structure of the Berger
nonlinearity (and does not seem possible for von Karman). We point
out that, in some sense, we obtain a stronger result than in the
analogous case for von Karman in \cite{springer}.
\begin{remark} The authors of \cite{springer} utilize additional assumptions on the
damping to obtain a unique continuation result that provides a
strict Lyapunov function for the dynamics. Then, they use an
indirect result (for a {\em gradient dynamical system}, asymptotic
smoothness implies the existence of a compact global
attractor---see the Appendix, Theorem \ref{0}); this is to
circumvent the need for direct estimation on the absorbing ball in
the case of non-rotational von Karman dynamics.\end{remark}

 \item Most importantly, by addressing non-rotational Berger dynamics
with boundary dissipation we provide results which seem to be
conspicuously missing in the mathematical literature. As we have
pointed out, the case of (HD) and (FCD) seem to be pathological,
in that the analysis of the Berger dynamics is not subsumed by
that of the von Karman.
\end{itemize}

\section{Abstract Plate Equations and Well-Posedness}\label{abplate}

The issue of well-posedness of strong or finite energy solutions
for nonlinear plates is recent. For a detailed and complete
discussion of general, abstract plate equations the reader is
referred to \cite{springer}. For clarity, let us recall some
definitions and results. The general  model arising in nonlinear
plate dynamics is:
\begin{equation}\label{abs_plate1} \begin{cases}
u_{tt}(t)+\mathcal{A}u(t)+\mathcal{A}G\overline{g}(G^{\ast }\mathcal{A}%
u_{t}(t))=F(u(t),u_{t}(t)),\text{ \ \ \ \ \ }t>0
\\
u(0)=u_{0},\text{ \ \ \ \ \ \ \ }u_{t}(0)=u_{1}.
\end{cases}
\end{equation}%

 \noindent To address the well-posedness for problem (\ref%
{abs_plate1}), we must give assumptions and properties of the
operators.

\begin{assumption}
\label{ass1} With reference to problem (\ref{abs_plate1}): \

\begin{itemize}
\item Let $\mathcal{A}$ be a closed, linear, positive, self-adjoint operator
acting on a Hilbert space $\mathcal{H}$, with $\mathscr{D}(\mathcal{A}%
)\subset \mathcal{H}.$

\item Let $U$ be another Hilbert space, and $U_{0}$ be a reflexive Banach
space, such that $U_{0}\subseteq U\subseteq U_{0}^{^{\prime }}.$
Additionally, assume that $\overline{g}:U_{0}\rightarrow
U_{0}^{^{\prime }}$
is a continuous mapping such that $\overline{g}(0)=0$ and%
\begin{equation*}
\left\langle \overline{g}(v_{1})-\overline{g}(v_{2}),v_{1}-v_{2}\right%
\rangle \geq 0,\text{ \ \ \ \ for all }v_{1},v_{2}\in U_{0}
\end{equation*}
where $\left\langle\cdot,\cdot\right \rangle$ denotes the scalar
product on $U$, or the duality pairing between $U_{0}$ and
$U_{0}^{^{\prime }}$.

\item The linear operator $G:U_{0}^{^{\prime }}\rightarrow \mathcal{H}$
has that $\mathcal{A}^{1/2}G:U_{0}^{^{\prime }}\rightarrow
\mathcal{H}$
\ is bounded (or equivalently, $G^{\ast }\mathcal{A}:\mathscr{D}(\mathcal{A}%
^{1/2})\rightarrow U_{0}$ is bounded, where the adjoint operator
$G^{\ast }$ is defined by the relation $\left\langle G^{\ast
}u,v\right\rangle =\left\langle u,Gv\right\rangle $).

\item Let V be a Hilbert space satisfying $\mathscr{D}(\mathcal{A}%
^{1/2})\subset V\subset \mathcal{H}\subset V^{^{\prime }}\subset \mathscr{D}(%
\mathcal{A}^{1/2})^{^{\prime }}$ with all injections being
continuous and dense. The nonlinear operator
$F:\mathscr{D}(\mathcal{A}^{1/2})\times
V\rightarrow V^{^{\prime }}$ is locally Lipschitz; that is:%
\begin{equation*}
\left\Vert F(u_{1},v_{1})-F(u_{2},v_{2})\right\Vert _{V^{^{\prime
}}}\leq
L(K)\left(\left\Vert \mathcal{A}^{1/2}(u_{1}-u_{2})\right\Vert _{\mathcal{H}%
}+\left\Vert v_{1}-v_{2}\right\Vert _{V}\right)
\end{equation*}%
for all $(u_{i},v_{i})\in \mathscr{D}(\mathcal{A}^{1/2})\times V$ such that $%
\left\Vert \mathcal{A}^{1/2}u_{i}\right\Vert
_{\mathcal{H}},\left\Vert v_{i}\right\Vert _{V}\leq K.$
\end{itemize}
\end{assumption}

\begin{remark}
In concrete applications, the term $\mathcal{A}G\overline{g}(G^{\ast }%
\mathcal{A}u_{t}$) models the boundary dissipation. Here, the
operator $G$ typically represents a suitable Green's map, and the
mapping $\overline{g}:U_{0}\rightarrow U_{0}^{^{\prime }}$ is a
Nemytskij-type operator (see Section 3 and 4).
\end{remark}

\noindent If we rewrite problem \eqref{abs_plate1} as a first
order equation we have the equivalent problem

\begin{equation*}
\frac{d}{dt}\mathbf u(t)+\mathbb A\mathbf
u(t)=\Big[\begin{array}{c}
                                                    0 \\
                                                    F(u(t),u_t(t)) \\
                                                  \end{array}\Big],~~
\mathbf u(0)=\mathbf u_0\equiv(u_0,u_1),
\end{equation*}
where $\mathbf u(t)=(u(t),u_t(t))$ and the operator $\mathbb
A:\mathscr{D}(\mathcal{A}^{1/2})\times V\rightarrow
\mathscr{D}(\mathcal{A}^{1/2})\times V$ is defined by
\begin{equation}\label{Atild}
\mathbb A=\Big[\begin{array}{cc}
                0 & -I \\
                \mathcal{A} & \mathcal{A}G\overline{g}G^*\mathcal{A} \\
              \end{array}\Big],
\end{equation}
where $\mathscr{D}(\mathbb{A})$ consists of elements $(x,y)\in
\mathscr{D}(\mathcal{A}^{1/2})\times\mathscr{D}(\mathcal{A}^{1/2})$
possessing the property
$$\mathcal{A}(x+G\overline{g}(G^*\mathcal{A}y))\in V^{^{\prime }}.$$
This structure of problem \eqref{abs_plate1} leads the following
definitions of strong and generalized solutions in this setting.

\begin{definition}
A function $u(t)$ is said to be \textbf{a strong solution} to (\ref%
{abs_plate1}) on a semi-interval $[0,T),$ iff

\begin{itemize}
\item $u(t)\in C([0,T);\mathscr{D}(\mathcal{A}^{1/2}))\cap C^{1}([0,T);V)$

\item $u\in W^{1,1}(a,b;\mathscr{D}(\mathcal{A}^{1/2}))$ and $u_{t}\in
W^{1,1}(a,b;V)$ for any $0<a<b<T$

\item $\mathcal{A[}u(t)+G\overline{g}(G^{\ast }\mathcal{A}u_{t}(t))]\in
V^{^{\prime }}$ for almost all $t\in \lbrack 0,T]$

\item Equation \eqref{abs_plate1} is satisfied in $V^{^{\prime }}$ for almost all $t\in
\lbrack 0,T]$

\item The initial condition in \eqref{abs_plate1} holds.
\end{itemize}
The function $u(t)$ is a strong solution on $[0,T]$ if, in
addition, we have that $(u(t),u_t(t))$ is continuous at $t=T$. A
function $u(t)$ is said to be \textbf{a generalized solution} to
\eqref{abs_plate1} on $[0,T]$, iff

\begin{itemize}
\item $u(t)\in C(0,T;\mathscr{D}(\mathcal{A}^{1/2}))\cap C^{1}(0,T;V)$

\item The initial condition in \eqref{abs_plate1} holds.

\item There exists sequences of strong solutions $\{u_{n}(t)\}$ to problem (%
\ref{abs_plate1}) defined on $[0,T]$ with initial data $%
(u_{0n},u_{1n})$ such that%
\begin{equation*}
\lim_{n \to \infty} ~\max_{t \in [0,T]}
\left\{ \left\Vert u_{n{t}}(t)-u_{t}(t)\right\Vert _{V}+\left\Vert \mathcal{%
A}^{1/2}(u_{n}(t)-u(t))\right\Vert _{\mathcal{H}}\right\} =0
\end{equation*}%
The function $u(t)$ is the generalized solution on a semi-interval
$[0,T),$ if $u(t)$ is a generalized solution on each subinterval
$[0,T^{^{\prime }}]\subset \lbrack 0,T).$
\end{itemize}
\end{definition}

The following theorem provides the existence and uniqueness of
local and global solutions. This is the primary result of Section
2.4 in \cite{springer}.

\begin{theorem}
\label{abs_wp} Under Assumption \ref{ass1}, the following hold:

\begin{itemize}
\item \textbf{Local strong solutions:} For every $u_{0},u_{1}\in \mathscr{D}(%
\mathcal{A}^{1/2})$ such that $\mathcal{A[}u_{0}+G\overline{g}(G^{\ast }%
\mathcal{A}u_{1})]\in V^{^{\prime }}$ there exists $t_{\max }>0$
and unique
strong solution such that $(u,u_{t})\in C([0,t_{\max });\mathscr{D}(\mathcal{%
A}^{1/2})\times V).$

\item \textbf{Local generalized solutions: }If $(u_{0},u_{1})\in \overline{%
\mathscr{D}(\mathbb{A})}\subseteq
\mathscr{D}(\mathcal{A}^{1/2})\times V$, where $\overline{%
\mathscr{D}(\mathbb{A})}$ is the closure of
$\mathscr{D}(\mathbb{A})$ in $\mathscr{D}(\mathcal{A}^{1/2})\times
V$.
Then there exists $t_{\max }>0$ and unique generalized solution such that $%
(u,u_{t})\in C\left([0,t_{\max
});\mathscr{D}(\mathcal{A}^{1/2})\times V\right).$

\item \textbf{Global solutions}: If, in addition, strong (or generalized)
solutions satisfy%
\begin{equation*}
\underset{t\in \lbrack 0,t_{\ast }]}{\sup }\left\{ \left\Vert \mathcal{A}%
^{1/2}u(t)\right\Vert _{\mathcal{H}}+\left\Vert
u_{t}(t)\right\Vert _{V}\right\} \leq M(t_{\ast },u_{0},u_{1})
\end{equation*}%
for every existence semi-interval $[0,t_{\ast }]$, then the local
solutions referred to above are global, which is to say $t_{\max
}=\infty .$
\end{itemize}
\end{theorem}

\section{Berger Model with (FCD) Conditions}\label{freeBC}

\subsection{Well-posedness---Proof of Theorem \protect\ref{mainresult0}}
We consider problem (\ref{plate}) with (FCD) boundary conditions.
The proof of Theorem \ref{mainresult0} is based
 on the application of the abstract setup above and Theorem
 \ref{abs_wp}.

\subsubsection{Application of Abstract Setup and Local Well-posedness}
Firstly, we need to modify the abstract model given above to fit
our case. For the functional setup, we introduce the following
spaces and operators:

\begin{itemize}
\item $\mathcal{H=}V\mathcal{\equiv }L^{2}(\Omega ),$ \ \ \ $U\equiv
L^{2}(\Gamma _{1})\times L^{2}(\Gamma _{1}),$

\item $U_{0}\equiv H^{1/2}(\Gamma _{1})\times H^{3/2}(\Gamma _{1}),$ \ \ \ $%
U_{0}^{^{\prime }}=$\ $H^{-1/2}(\Gamma _{1})\times H^{-3/2}(\Gamma
_{1}),$

\item $\mathcal{A}u\equiv \Delta ^{2}u,$ $u\in \mathscr{D}(\mathcal{A}),$
where%
\begin{equation*}
\mathscr{D}(\mathcal{A})\equiv \left\{ u\in H^{4}(\Omega ):%
\begin{array}{c}
u=\nabla u=0\text{ on }\Gamma _{0} \\
\Dn\Delta u+(1-\mu )B_{2}u-\mu_{1}u=0\text{ on }%
\Gamma _{1} \\
\Delta u+(1-\mu )B_{1}u=0\text{ on }\Gamma _{1}%
\end{array}%
\right\}
\end{equation*}

\item $V^{^{\prime }}=L^{2}(\Omega )$ and $\mathscr{D}(\mathcal{A}%
^{1/2})=H_{\Gamma _{0}}^{2}(\Omega )$ which is given in Section 1.3. %

\item $F(u,v)=-f_B(u)+p$, ~with ~$f_B(u)=(\gamma -\left\Vert \nabla u\right\Vert
^{2})\Delta u,$

\item $G:U_{0}^{^{\prime }}\rightarrow \mathcal{H}$ , $G(g_{1},g_{2})\equiv
G_{1}(g_{1})-G_{2}(g_{2})$ where $G_{i},$ $i=1,2$, are defined by%
\begin{equation*}
G_{1}(v)\equiv u\text{ iff }\left\{
\begin{array}{c}
\Delta ^{2}u=0\text{ in }\Omega ,\text{ }u=\nabla u=0\text{ on
}\Gamma _{0}
\\
\Delta u+(1-\mu )B_{1}u=v\text{ on }\Gamma _{1} \\
\Dn\Delta u+(1-\mu )B_{2}u=\mu_{1}u\text{ on }\Gamma
_{1},%
\end{array}%
\right.
\end{equation*}%
and%
\begin{equation*}
G_{2}(v)\equiv u\text{ iff }\left\{
\begin{array}{c}
\Delta ^{2}u=0\text{ in }\Omega ,\text{ }u=\nabla u=0\text{ on
}\Gamma _{0}
\\
\Delta u+(1-\mu )B_{1}u=0\text{ on }\Gamma _{1} \\
\Dn\Delta u+(1-\mu )B_{2}u=\mu_{1}u+v\text{ on }%
\Gamma _{1}.%
\end{array}%
\right.
\end{equation*}

\item The mapping $\overline{g}:U_{0}\rightarrow U_{0}^{^{\prime }}$ has the form%
\begin{equation*}
\overline{g}:(v_{1},v_{2})\rightarrow (g_{1}(v_{1}),
g_{2}(v_{2})),\text{ \ \ \ }(v_{1},v_{2})\in U_{0}
\end{equation*}%

\end{itemize}
\noindent With the above notation, taking into account that in our case: $g_{1}=0$ and $%
g_{2}(v)=|v|^6 v$,   $G(g_{1},g_{2})\equiv
-G_{2}(g_{2})$. We reduce the equation in (\ref{plate}) to the following form%
\begin{equation*}
u_{tt}+\mathcal{A}\left[u-G_{2}\left(\left[|u_{t}|^6 u_t\right]|_{\Gamma
_{1}}\right)\right]=F(u(t),u_{t}(t)),\text{ \ \ \ \ \ }t>0,
\end{equation*}%
or equivalently,
\begin{equation}
u_{tt}+\mathcal{A[}u+G\overline{g}(G^*\mathcal{A}u_t
)]=F(u(t),u_{t}(t)),\text{ \ \ \ \ \ }t>0.  \label{4.4}
\end{equation}%
 Now, before giving the well-posedness result for
(\ref{4.4}), we verify that the conditions in Assumption
\ref{ass1} are satisfied.

\begin{lemma}
\label{oper} With reference to (\ref{4.4}), operators
$\mathcal{A},G$ and $F$ introduced above comply with the
requirements of Assumption \ref{ass1}.
\end{lemma}

\begin{proof}
By definition of $\mathcal{A}$ as above, \ $\mathcal{A}$ is
closed, linear, positive, self-adjoint, densely defined on $\mathcal{H}$ with $\cD(\mathcal{A}%
^{1/2})=H_{\Gamma _{0}}^{2}(\Omega )$, and we have dense,
continuous injections:%
\begin{equation*}
\cD(\mathcal{A}^{1/2})\subset V=\mathcal{H}=V^{^{\prime }}\subset \cD(\mathcal{A}%
^{1/2})^{^{\prime }}=[H^{2}_{\Gamma_0}(\Omega)]'.
\end{equation*}%
The Nemytskij operator%
\begin{equation*}
(g_{1}, g_{2})=(0,g_{2}):U_{0}\rightarrow U_{0}^{^{\prime }}
\end{equation*}%
is monotone \cite{springer,showalter} and continuous, since
$H^{3/2}(\Gamma_1)\subset C(\Gamma_1)$ by the Sobolev embeddings.

\noindent The requirement on $G$ in Assumption \ref{ass1} reduces
the fact that
\begin{equation}
G^{\ast }\mathcal{A}^{1/2}:\mathcal{H}\rightarrow U_{0}\equiv
H^{1/2}(\Gamma _{1})\times H^{3/2}(\Gamma _{1})\label{opG}
\end{equation}
is bounded. An application of Green's theorem with the
calculations given in \cite[Section 3.2.2]{springer} yield that
the
mappings%
\begin{equation*}
G_{1}:L^{2}(\Gamma _{1})\rightarrow \mathscr{D}(\mathcal{A}^{5/8-\epsilon }),%
\text{ \ \ }G_{2}:L^{2}(\Gamma _{1})\rightarrow \mathscr{D}(\mathcal{A}%
^{7/8-\epsilon })
\end{equation*}%
\begin{equation*}
G_{1}^{\ast }\mathcal{A}:\mathscr{D}(\mathcal{A}^{1/2})\rightarrow
H^{1/2}(\Gamma _{1}),\text{ \ \ }G_{2}^{\ast }\mathcal{A}:\mathscr{D}(%
\mathcal{A}^{1/2})\rightarrow H^{3/2}(\Gamma _{1})\text{\ \ \ }
\end{equation*}%
are bounded and for every $u\in
\mathscr{D}(\mathcal{A}^{3/8+\epsilon
})\subset H^{3/2+4\epsilon }(\Omega )$.\ We then have that%
\begin{equation*}
G^{\ast }\mathcal{A}u=(G_{1}^{\ast }\mathcal{A}u, -G_{2}^{\ast }\mathcal{A}%
u)=(\Dn u|_{\Gamma _{1}}, u|_{\Gamma _{1}}).
\end{equation*}%
By the last equality, \eqref{opG} translates into the fact that
\begin{equation*}
u\rightarrow G^{\ast }\mathcal{A}^{1/2}u=(\Dn
\mathcal{A}^{-{1/2}}u, \mathcal{A}^{-{1/2}}u)
\end{equation*}
is bounded from $L^2(\Omega)$ into $H^{1/2}(\Gamma _{1})\times
H^{3/2}(\Gamma _{1})$. This follows from the identification
$\cD(\mathcal{A}^{1/2})=H_{\Gamma _{0}}^{2}(\Omega )$, and the
trace theorem, which imply that the maps
\begin{equation*}
\Dn(\cdot):H^2(\Omega)\rightarrow H^{1/2}(\Gamma),\text{ \ \
}(\cdot)|_{\Gamma}:H^2(\Omega)\rightarrow H^{3/2}(\Gamma)
\end{equation*}
are bounded. The  validity of requirement on $G$ is thus observed.

 Now, we need to verify that the operator
$F:\cD(\mathcal{A}^{1/2})\times V\rightarrow V{^{\prime }}(=V)$,
defined by $$F(u,v)=-f_B(u)+p,$$ is locally Lipschitz. Let $u,w \in
\cD(\mathcal{A}%
^{1/2})$ with~
$||A^{1/2}u||_{\mathcal{H}}, ||A^{1/2}w||_{\mathcal{H}}\leq R$,~
where the bilinear form $a(\cdot,\cdot)$ plays a special role for the norm
in $\mathscr H$ under (FCD) and is equivalent to $||\Delta
\cdot||_0$ for (HD)\footnote{For (HD), this follows from
the well-posedness of the biharmonic problem associated to the
these boundary conditions.}:
\begin{align*}
||f_B(u)-f_B(w)||_V \le
C||u-w||_{\cD(\cA^{1/2})}+\big|\big|~||\nabla u||^2\Delta u -
||\nabla w||^2\Delta w\big|\big|.
\end{align*}
Now,
\begin{align*}
\big|\big|~||\nabla u||^2\Delta u - ||\nabla w||^2\Delta
w\big|\big| \le & \big|\big|~||\nabla u||^2\Delta u-||\nabla
u||^2\Delta w+||\nabla u||^2\Delta w - ||\nabla
w||^2\Delta w\big|\big|\\[.1cm]
\le & \big|\big|~||\nabla u||^2\Delta(u-w)+\Delta w(||\nabla u||^2-||\nabla w||^2)\big|\big|\\[.1cm]
%\le &  ||\nabla u||^2||\Delta(u-w)||_{0,\Omega}+||\Delta w||_{0,\Omega}(||\nabla u||^2-||\nabla w||^2)\\
\le &  ~C||u||_{1,\Omega}^2||u-w||_{2,\Omega}+C||w||_{2,\Omega}\big(||\nabla u||+||\nabla w||\big)\big(||\nabla u||-||\nabla w||\big)\\[.1cm]
\le & ~C||u||_{1,\Omega}^2||u-w||_{2,\Omega}+C||w||_{2,\Omega}\big(||\nabla u||+||\nabla w||\big)\big(||u-w||_{1,\Omega}\big)\\[.1cm]
%\le & ||u||_{\cD(A^{1/2})}^2||u-w||_{\cD(A^{1/2})}\\
%&+||w||_{\cD(A^{1/2})}\big(||u||_{\cD(A^{1/2})}+||w||_{\cD(A^{1/2})}\big)||u-w||_{2,\Omega}\\
\le & ~C\left(R \right)||u-w||_{\cD(A^{1/2})},
\end{align*}
where we have freely used the Poincare's inequality, the equivalence $H^2(\Omega) \equiv \mathscr D(A^{1/2})$, and Sobolev
embeddings. The locally Lipschitz property for $F$ is thus
observed.

Moreover, for initial data $(u_0,u_1)\in \mathscr{H}$, since the
set
\begin{equation*}
\mathcal{L}=(H_{\Gamma _{0}}^{2}\cap H^4)(\Omega )\times
C^{\infty}_0(\Omega)
\end{equation*}
is dense in $\mathscr{H}$ and belongs to the domain of
corresponding operator $\mathbb{A}$ given in \eqref{Atild}, the
application of Theorem \ref{abs_wp} gives the local (in time)
existence and uniqueness of generalized solutions.
\end{proof}

\noindent\textbf{\textit{A Priori Bound and Global
Well-posedness---Completion of proof of Theorem
\ref{mainresult0}}}

\vspace{0.25cm} \noindent Having obtained local generalized
solutions, our aim is to obtain {\em global a priori estimates}
for these solutions. We note that our computations are first
performed on strong solutions, and then extended via density to
$\mathscr{H}$ in the limit. Our starting point is the energy
relation (\ref{e-ineq}), which follows via integration by parts
for strong solutions and remains valid for generalized solutions.
It can easily be shown \cite{ch-0} that there exists positive
constants $c_0,c_1$ and $C$ such that
\begin{equation}
c_0\mathscr{E}(u,u_t)-C\leq E(u,u_t)\leq c_1\mathscr{E}(u,u_t)+C
\label{energy}
\end{equation}
for all generalized solutions $(u,u_t)\in \mathscr{H}$.\footnote{For more details, see the
discussion preceding \eqref{nonen}.}
Now, using (\ref{energy}) in the energy identity (\ref{e-ineq}), we have%
\begin{equation*}
E(t)+\int\limits_{0}^{t}\int\limits_{\Gamma _{1}}|u_{t}|^8
+\int\limits_{0}^{t}\int\limits_{\Gamma _{1}}(\gamma -\left\Vert
\nabla u\right\Vert ^{2})\left(\partial _{\nu }u\right)u_{t} \leq
E(0)+C
\end{equation*}%

Dealing with the non-dissipative term above requires some
 calculation. By the H\"{o}lder inequality we have
\begin{align*}
\left|\int\limits_{0}^{t}\int\limits_{\Gamma _{1}}(\gamma
-\left\Vert \nabla u\right\Vert ^{2})(\partial _{\nu
}u)(u_{t})\right| \leq&~ C\left\{ \int\limits_{0}^{t}\gamma
\left\Vert \Dn u\right\Vert _{L^2(\Gamma_1)}\left\Vert
u_{t}\right\Vert _{L^{2}(\Gamma _{1})}
+\int\limits_{0}^{t}\left\Vert \nabla u\right\Vert ^{2}\left\Vert
\Dn u\right\Vert _{L^2(\Gamma_1)}\left\Vert u_{t}\right\Vert
_{L^{2}(\Gamma _{1})}\right\}.
\end{align*}
Now, applying the {\em trace moment inequality} for $\Dn u$, \cite[Theorem 1.6.6]{tracemoment} we
have
\begin{align*}
\left|\int\limits_{0}^{t}\int\limits_{\Gamma _{1}}(\gamma
-\left\Vert \nabla u\right\Vert ^{2})(\partial _{\nu
}u)u_{t}\right| \leq C \left\{\int\limits_{0}^{t} \gamma\left\Vert
 u\right\Vert_2 ^{1/2}\left\Vert \nabla u\right\Vert
^{1/2}\left\Vert u_{t}\right\Vert _{L^{2}(\Gamma _{1})}
+\int\limits_{0}^{t}\left\Vert  u\right\Vert_2
^{1/2}\left\Vert \nabla u\right\Vert ^{5/2}\left\Vert
u_{t}\right\Vert _{L^{2}(\Gamma _{1})}\right\}.
\end{align*}
We note the equivalence of ~$a(\cdot,\cdot)$ ~and ~$||\cdot||_2$. We then use Young's inequality for the first integral on the RHS,
with powers 2, 4, 8, and 8, so ~~$(1/2)+(1/4)+(1/8)+(1/8)=1.$  For
the second integral of RHS, with powers 4, 8/5, 8, where
~~$(1/4)+(5/8)+(1/8)=1,$~ we obtain
\begin{equation*}
\left|\int\limits_{0}^{t}\int\limits_{\Gamma _{1}}(\gamma
-\left\Vert \nabla u\right\Vert ^{2})(\partial _{\nu}u
)u_{t}\right|\leq
C_{\epsilon}\gamma^2t+C_{\epsilon}\int\limits_{0}^{t}E(s)ds+\epsilon
\int\limits_{0}^{t}\left\Vert u_{t}\right\Vert _{L^{2}(\Gamma
_{1})}^{8}.
\end{equation*}
This yields

\begin{equation*}
E(t)\leq C(\gamma,t,E(0))+C\int\limits_{0}^{t}E(s)ds.
\end{equation*}%
An application of Gronwall's inequality gives:
\begin{equation*}
E(t)\leq C(T)\text{, \ \ \ \ for all }t\in \lbrack 0,T].
\end{equation*}%
This implies the global existence of strong (and thus generalized)
solutions to equation (\ref{4.4}). Since these solutions also solve the original problem---(%
\ref{plate}) with (FCD) boundary conditions---the proof of Theorem \ref%
{mainresult0} is now completed.

\begin{remark}
At the present time, we do not see how the global well-posedness
of solutions can be established in the absence of such highly
nonlinear damping. As discussed above, this is an indication that
the dynamical systems properties of solutions does not follow
readily from the equations via standard techniques used at the
finite-energy level. Without highly nonlinear damping active on
the boundary, we may not have control of the non-dissipative term
$$\displaystyle \int_0^t(\gamma-||\nabla
u||^2)\int_{\Gamma_1}(\partial_{\nu} u) u_t ~d\Gamma dt.$$ Additionally, since we have no control over $E(t)$ as ~$T\to +\infty$, long-time behavior analysis is not viable.
\end{remark}

\section{Berger Model with (HD) Conditions}
In this section we consider problem (\ref{plate}) with (HD)
boundary conditions.
\subsection{Well-posedness---Proof of Theorem \protect\ref{wellpHD}}
Similar to Section \ref{freeBC}, the proof of Theorem
\ref{wellpHD} is based on the abstract setup given in Section
\ref{abplate} and the application of Theorem \ref{abs_wp}. We
adapt the abstract model (\ref{abs_plate1}) to this case.

\subsubsection{Application of Abstract Setup and Verification of Hypotheses}
We introduce the following spaces and operators:

\begin{itemize}
\item $\mathcal{H=}L^{2}(\Omega ),$ \ \ \ $U\equiv
L^{2}(\Gamma)$

\item $U_{0}\equiv H^{1/2}(\Gamma)$ \ \ \ $%
U_{0}^{^{\prime }}=$\ $H^{-1/2}(\Gamma)$

\item $\mathcal{A}u\equiv \Delta ^{2}u,$ $u\in \mathscr{D}(\mathcal{A}),$
where%
\begin{equation*}
\mathscr{D}(\mathcal{A})\equiv \left\{ u\text{ in } (H^{4}\cap H_{0}^{1})(\Omega ):%
\begin{array}{c}
\Delta u=0\text{ on } \Gamma
\end{array}%
\right\}
\end{equation*}

\item $V^{^{\prime }}=V=L^{2}(\Omega )$ and $\mathscr{D}(\mathcal{A}%
^{1/2})=(H_{0}^{1}\cap H^{2})(\Omega )$.

\item $F(u,v)=-f_B(u)+p$ where $f_B(u)=(\gamma -\left\Vert \nabla u\right\Vert
^{2})\Delta u.$

\item $G:L^{2}(\Gamma)\rightarrow L^{2}(\Omega )$ denotes a biharmonic extension of the boundary values defined on $\Gamma$. That is:
\begin{equation*}
G(v)\equiv u\text{ iff }\left\{
\begin{array}{c}
\Delta ^{2}u=0\text{ in }\Omega,
\\
u=0\text{ on
}\Gamma, \\
\Delta u=v \text{ on }\Gamma.%
\end{array}%
\right.
\end{equation*}%

\item The mapping $\overline{g}:U_{0}\rightarrow U_{0}^{^{\prime }}$
is determined by the function $D(\cdot)$ according to the formula
\begin{equation*}
(\overline{g}(v),w)_{U_0,U_{0}^{^{\prime }}}=\int_{\Gamma}
D(v)wd\Gamma ~~~~~ v,w\in U_0.
\end{equation*}%
\end{itemize}

We note that the application of Green's formula gives that
$G^{\ast}\mathcal{A}=\frac{\partial}{\partial \nu}\big|_{\Gamma}$,
and hence
$\ker[G^{\ast}\mathcal{A}]\cap \mathscr{D}(\mathcal{A}%
^{1/2})$ contains $H_0^2(\Omega)$ and is dense in $L^2(\Omega)$.
This implies that for all initial data $(u_0,u_1)\in \mathscr{H}$
the set $\mathscr{D}(\mathcal{A})\times (\ker[G^{\ast}\mathcal{A}]\cap \mathscr{D}(\mathcal{A}%
^{1/2}))$ is dense in $\mathscr{H}$ and belongs to the domain of
the corresponding operator $\mathbb A$.

In line with the above setup, the verification of the requirements
of Assumption \ref{ass1} follows as  the case of (FCD) damping.
For more details the reader is referred to \cite[Section
4.2]{springer}. The proof of Theorem \ref{wellpHD} is obtained by
the application of Theorem \ref{abs_wp} via the density argument
above.

\subsection{Long Time Behavior Under (HD)---Proof of Theorem \protect\ref{gatt}}
This section is devoted to study the long time behavior of
dynamical system generated by the solution of \eqref{plate} taken
with (HD) boundary conditions. Our principal goal is to prove
Theorem \protect\ref{gatt}, the existence of a compact global
attractor for the dynamics.

We utilize the notations consistent with the case of (FCD), but
now specified to the case of (HD): $\displaystyle
E(t)=\frac{1}{2}\big\{ ||u_{t}||^{2}+||\Delta u||^2\big\}$, as
well as the identical definitions for $\widehat{E}$ and $\mathscr
E$ from before.

As in the proof of \cite[Lemma 1.5.4]{springer}, for all $(u,u_t)\in \mathscr{H}=(H^{2}\cap H_{0}^{1})(\Omega )\times L^2(\Omega)$, and any $%
\epsilon >0$, it is straightforward to show the {potential energy
bound} (as we will refer to it below):
\begin{equation}
\left|\frac{\gamma}{2}||\nabla u(t)||^{2}+(p,u)\right| \leq \epsilon \left[ ||\Delta u(t)||^{2}+%
\frac{1}{2}||\nabla u(t)||^{4}\right] +M(\epsilon,\gamma,p).
\label{pibound}
\end{equation}%
\begin{remark}
In fact, as in \cite{springer}, the key bound which implies 
\eqref{pibound} is
\begin{equation}\label{lot}
||u||_{2-\eta}^2 \le \epsilon \left[ ||\Delta u(t)||^{2}+%
\frac{1}{2}||\nabla u(t)||^{4}\right]
+M(\epsilon),~~\eta,\epsilon>0.
\end{equation}
\end{remark}
\noindent This also yields the energetic bounds
\begin{equation}
c_0\widehat{E}(u,u_t)-C\leq \mathscr{E}(u,u_t)\leq
c_1\widehat{E}(u,u_t)+C, \label{nonen}
\end{equation}%
for some $c_0,c_1,C>0$ depending on $p$ and $\gamma. $
Accordingly, we introduce  more notation for the study of
long-time behavior (following \cite{c-e-l}):
\begin{equation}
0 \le \mathscr E_M(t) \equiv \mathscr E(t)+M,\label{not1}
\end{equation}
where $M= M(\epsilon,\gamma,p)$ is the constant given in
\eqref{pibound}.

\subsubsection{Overview of the Proof}
Proof of Theorem \ref{gatt} (which gives the existence of a
compact global attractor for $(S_H,\mathscr H)$) follows from the
application of the abstract Theorem \ref{0} (in the Appendix).
Utilizing Theorem \ref{0} requires showing that the dynamical
system generated by solutions to \eqref{plate} taken with (HD)
conditions is {\em dissipative} (i.e., possess a bounded absorbing
set), and has the {\em asymptotic smoothness} property. Thus, to
prove Theorem \ref{gatt}, we will first present the dissipativity
analysis of \eqref{plate} in Section \ref{sec:abs}. We will
utilize energy methods to show the existence of the absorbing ball
that attracts every trajectory after a certain time. This will be
given as Theorem \ref{abs-ball}. Then, in Theorem \ref{asympt},
the asymptotic smoothness property of the dynamical system is
given. Its proof follows through several steps in Section
\ref{smoothnesssec}. With these two results, the existence of a
compact global attractor for $(S_H,\mathscr H)$ will be
demonstrated.
\begin{remark} We will show an explicit
bound on the size of the absorbing ball. We note that this is in
contrast to the approach taken in \cite{springer} for the
non-rotational von Karman plate with hinged-type boundary
dissipation; there, the dynamics (under additional assumptions,
including requisite frictional damping in the interior) are shown to
be gradient, and thus the existence of a compact global attractor
(roughly) follows from the asymptotic smoothness property (see
Theorem \ref{gradsmooth} in the Appendix). Also see the previous
discussion following Theorem \ref{gatt}. \end{remark}

\subsubsection{Explicit Estimates on Absorbing Ball}\label{sec:abs}
In this section, as the first step of the proof of Theorem
\ref{gatt}, we show the dissipativity of the dynamical system
corresponding to generalized solutions of \eqref{plate} with (HD)
boundary conditions. To this end, we give an explicit estimate on
the absorbing ball for which we utilize multiplier techniques and
sharp trace results for the linear plate equation with a given RHS
\cite{sharptrace}. We will show the following theorem:

\begin{theorem}
\label{abs-ball} Let Assumption \ref{ass2} and \ref{geom} hold.
Then there exists an absorbing set $\mathscr{B}\subset
\mathscr{H}$ for generalized solutions to (\ref{plate}) taken with
(HD) boundary conditions. This is to say: for all $R_0>0$ and
initial data $(u_0,u_1)\in \mathscr{H}$ with
$||(u_0,u_1)||_{\mathscr{H}}\leq R_0$, there exists a $t_0
=t(R_0)$ such that $(u(t),u_t(t))\in \mathscr{B}$ for $t\geq t_0.$
\end{theorem}

\noindent In order to prove this central theorem, the main
technical ingredient will be the following {``observability"
estimate}, stated below.

\begin{lemma}\label{obsineq}
Let Assumption \ref{geom} is in force. Let $T>0$ and
$0<\alpha<T/2$. Then any solution to \eqref{plate} taken with
(HD) boundary conditions satisfies the following estimate:
\begin{equation*}
(T-2\alpha)\widehat E(T)\leq C \widehat{E}(0)+C(T)||D(\Dn
u_t)||^2_{L^2(0,T;L^2(\Gamma))}\Big[\int_0^T(\widehat E(\tau))^2
d\tau \Big]
\end{equation*}
\begin{equation}
+C(T)\left[||D(\Dn u_t)||_{L^2(0,T;L^2(\Gamma~))}^2+||\Dn
u_t||_{L^2(0,T;L^2(\Gamma))}^2\right]+C(T).\label{en*}
\end{equation}
\end{lemma}

\begin{proof}[Proof of Lemma \ref{obsineq}]
Now, let $u$ be a generalized solution to \eqref{plate} taken with
(HD) boundary conditions. We begin with the equipartition
multiplier
$u$ and the flux multiplier $h\cdot \nabla u$, where $h(\mathbf{x})=\mathbf{x}-\mathbf{x}_{0}$ for some appropriately chosen $%
\mathbf{x}_{0}\in \mathbb{R}^{2}$ (satisfying Assumption \ref{geom}). Multiplying \eqref{plate} by
$u$ and $h\cdot\nabla u$, respectively, integrating over the
space-time cylinder, and using Green's theorem we have
\begin{equation}
\int_0^T\left\{ ||\Delta u||^{2}+||\nabla
u||^{4}-||u_{t}||^{2}\right\} +\int_0^T \int_{\Gamma~}D(\partial
_{\nu }u_t)(\partial _{\nu
}u)=-(u_{t},u)\big|_{0}^{T}+\int_0^T\left\{ (p,u)+\gamma ||\nabla
u||^{2}\right\},
 \label{mult1}
\end{equation}%
and
\begin{align}
 \int_0^T&\left\{ ||\Delta u||^{2}+||u_{t}||^{2}\right\} +\int_0^T \int_{\Gamma }\left\{
\frac{1}{2}|\Delta u|^{2}(h\cdot \nu )+\partial _{\nu }(\Delta
u)(h\cdot \nabla u)-(\Delta u)\partial _{\nu }(h\cdot \nabla
u)\right\}\nonumber \\
&+\dfrac{1}{2}\int_{0}^{T}(\gamma -||\nabla
u||^{2})\int_{\Gamma }(h\cdot \nu )|\partial _{\nu }u|^{2}
=-(u_{t},h\cdot \nabla u)%
\big|_{0}^{T}+ \int_{0}^{T} (p,h\cdot \nabla u). \label{mult2}
\end{align}
 We have used the facts:
\begin{align*}
\Big((\gamma-||\nabla u||^2)\Delta u, h\cdot \nabla u\Big)
=&~(||\nabla u||^2-\gamma)\Big(\nabla u,\nabla(h \cdot \nabla
u)\Big)+\int_{0}^T[\gamma-||\nabla u||^2]\int_{\Gamma~}
(\partial_{\nu} u)(h\cdot \nabla u)\\
\Big(\nabla u,\nabla(h \cdot \nabla u)\Big) =
&~\dfrac{1}{2}\text{div} \big(h|\nabla
u|^2\big)~~~~\text{and}~~~~h \cdot \nabla u = (h\cdot
\nu)\partial_{\nu} u~~\text{on}~~\Gamma,
\end{align*}
owing to the divergence theorem, and the zero Dirichlet condition
on the whole of $\Gamma$ (for (HD)). Now, taking a suitable
combination of \eqref{mult1} and \eqref{mult2} we obtain:
\begin{align}
\int_{0}^T\big\{3||\Delta u||^{2}& +||\nabla
u||^{4}+||u_{t}||^{2}-\gamma ||\nabla u||^{2}\big\}
 +\int_{0}^{T}(\gamma -||\nabla u||^{2})\int_{\Gamma}(h\cdot
\nu )|\partial _{\nu }u|^{2} \notag   \\
=& -\big\{(u_{t},u)\big|_{0}^{T}+2(u_{t},h\cdot \nabla u)\big|_{0}^{T}\big\}+ \int_{0}^T (p,u)+2\int_{0}^T (p,h%
\cdot \nabla u)-\int_0^T \int_{\Gamma~}D(\partial _{\nu }
u_t)(\partial
_{\nu }u)  \notag \\
& +\int_0^T \int_{\Gamma }\big\{2(\Delta u)\partial _{\nu }(h\cdot
\nabla u)-(h\cdot \nu )|\Delta u|^{2}-2\partial _{\nu }(\Delta
u)(h\cdot \nabla u)\big\}.\label{prelimid}
\end{align}%
Using the bound on the nonlinear potential energy \eqref{pibound} and reorganizing the terms above,
we have a preliminary estimate:
\begin{align}
c\int_{0}^{T}\widehat{E}(\tau )\leq \notag
\\
\mathbf{(T1)}~~ & -\Big\{(u_{t},u)\big|_{0}^{T}+2(u_{t},h\cdot \nabla u)\big|_{0}^{T}\Big\}+ \int_{0}^T (p,u)+2\int_{0}^T (p,h%
\cdot \nabla u)+C(\gamma)\cdot T \notag \\
\mathbf{(T2)}~~& -2\int_{0}^T \int_{\Gamma }\partial _{\nu
}(\Delta u)(h\cdot \nabla u)-\int_{0}^{T}(\gamma -||\nabla
u||^{2})\int_{\Gamma}(h\cdot \nu
)|\partial _{\nu }u|^{2} \notag\\
%&-\int_{0}^T\int_{\Gamma_0}(h\cdot \nu )|\Delta u|^{2} +\int_{0}^T\int_{\Gamma_0}2(\Delta u)\partial_{\nu}(h\cdot \nabla u) \notag\\
\mathbf{(T3)}~~& -\int_0^T \int_{\Gamma~}D(\partial _{\nu }
u_t)(\partial _{\nu }u)+\int_{0}^T\int_{\Gamma~}2(\Delta
u)\partial _{\nu }(h\cdot \nabla
u)-\int_0^T\int_{\Gamma} (h\cdot \nu )|\Delta u|^{2}
\label{prelimid'}
\end{align}
Firstly, by \eqref{pibound} and Young's inequality, we observe
\begin{equation}
|\textbf{(T1)}| \le C\widehat E(0)+\epsilon \left[\widehat
E(T)+\int_0^T\widehat
E(\tau)d\tau\right]+C(p,\epsilon,\gamma)\cdot T.
\end{equation}

Now, employing the geometric Assumption \ref{geom} on ({\bf
T2})--({\bf T3}), dropping the negatively signed terms on the RHS,
combining our computations thus far and reorganizing, we obtain
the following intermediate relation:
\begin{align}
c\int_{0}^{T}\widehat{E}(\tau )\leq & ~C\widehat{E}(0)+\epsilon
\left[\widehat
E(T)+\int_0^T\widehat E(\tau)\right]+C(p,\epsilon,\gamma)\cdot T  \notag\\
& -2\int_{0}^T \int_{\Gamma}\partial _{\nu }(\Delta u)(h\cdot
\nabla u)-\int_{0}^{T}(\gamma -||\nabla u||^{2})\int_{\Gamma
}(h\cdot \nu
)|\partial _{\nu }u|^{2} \notag\\
& -\int_0^T \int_{\Gamma~}D(\partial _{\nu } u_t)(\partial _{\nu
}u)+\int_{0}^T\int_{\Gamma~}2(\Delta u)\partial _{\nu }(h\cdot
\nabla u).
\label{energy2}
\end{align}%

\noindent In order to estimate the RHS of the above
inequality---which involves higher order trace terms---we follow
the analysis of \cite{ji} (which itself critically relies on
Theorem \ref{sharpt} from \cite{sharptrace} given below). However,
 we approach the higher order trace term involving $\partial _{\nu
}(\Delta u)$  in a more straightforward way which, ultimately,
benefits our analysis and fundamentally exploits the structure of
the Berger nonlinearity to obtain cancellations.

Now, let the operator $A=\Delta$, acting on $L^2(\Omega)$ with
domain $H^2(\Omega)\cap H_{0}^1(\Omega)$, and let $D_L$ be the
associated Dirichlet ``lift" map defined by
$$D_L g = f \in L^2(\Omega) ~\iff~ \Delta f = 0~\text{ in}~\Omega ~\text{ and }~
f= g~\text{ on }~ \Gamma.$$ Let $F\equiv\Big\{-f_B(u)+p\Big\}.$
Accounting for boundary conditions, the operator representation of
\eqref{plate} is then
\begin{equation}\label{thiss}
 {u}_{tt}+A\left[A {u}-D_L\left(\Delta
 {u}\big|_{\Gamma}\right)\right]=F.
\end{equation}
Applying $A^{-1}$ to \eqref{thiss} (justified on strong solutions,
and a posteriori on generalized solutions via the corresponding estimate) we obtain
\begin{equation*}
\Delta  {u} = A  {u} = -A^{-1}
 {u}_{tt}+D_L\left(\Delta
 {u}\big|_{\Gamma}\right)+A^{-1} F ~~\text{in}~~
\mathscr D'(\Omega).
\end{equation*}

\noindent Now, taking the normal derivative of both sides of above
equality, multiplying by $h\cdot \nabla u$, integrating over
$[0,T] \times \Gamma$, and reading off from the equation we have
the relation
\begin{align}
\int_{0}^{T}\int_{\Gamma}\Dn(\Delta   u)(h\cdot \nabla u) =&~
-\int_{0}^{T}\int_{\Gamma}\Dn (A^{-1} u_{tt})(h\cdot\nabla
u)+\int_{0}^{T}\int_{\Gamma~}\Dn D_L\left(\Delta
u\big|_{\Gamma}\right)(h\cdot \nabla
 u)\nonumber\\ &+\int_{0}^{T}\int_{\Gamma} (\Dn A^{-1}F)
(h\cdot \nabla u)\label{thisrelation}.
\end{align}

\noindent In order to estimate each term of the RHS of
\eqref{thisrelation} we use the approach in
 \cite{ji}. The following is the {\em critical step} which
allows our approach to the higher order trace terms to obtain. {\em This relies critically
on the structure of the Berger nonlinearity.} If we note that
\begin{align}
A^{-1}F=&~ (\Delta_{D_L})^{-1}\big[||\nabla
u||^2-\gamma\big]\Delta
u+A^{-1}p \notag\\[.2cm]
=&~\big[||\nabla u||^2-\gamma\big] u+A^{-1} p, \label{relation2}
\end{align}
and again $h\cdot \nabla u =(h\cdot \nu) \Dn u$, since
$u=0$ on $\Gamma$; we have:
\begin{align}\nonumber
-2\int_{0}^{T}&\int_{\Gamma}\Dn(\Delta   u)(h\cdot \nabla
u)-\int_{0}^{T}(\gamma -||\nabla u||^{2})\int_{\Gamma}(h\cdot
\nu )|\partial _{\nu }u|^{2}\\=&~ 2\int_{0}^{T}\int_{\Gamma}\Dn
(A^{-1} {u}_{tt})(h\cdot \nabla u)-2\int_{0}^{T}\int_{\Gamma~}\Dn
D_L\left(\Delta
{u}\big|_{\Gamma}\right)(h\cdot \nabla u) \nonumber\\
&+2\int_{0}^{T}(\gamma-||\nabla u||^2) \int_{\Gamma}(h\cdot
\nu)(\Dn u)^2-2\int_{0}^{T} \int_{\Gamma} \Dn(A^{-1}
p(\xb))(h\cdot \nabla u)\nonumber\\
&-\int_{0}^{T}(\gamma -||\nabla u||^{2})\int_{\Gamma}(h\cdot
\nu )|\partial _{\nu }u|^{2}
\end{align}
\begin{remark}
The cancellation in the nonlinear terms above is {\em critical} to the argument below. We
will be able to use the sign of the term and the standard star-shaped geometric condition
Assumption \ref{geom} on $\Gamma~$ to discard the nonlinear
boundary contribution from our energy approach.
\end{remark}
\begin{align}\nonumber
-2\int_{0}^{T}&\int_{\Gamma}\Dn(\Delta   u)(h\cdot \nabla
u)-\int_{0}^{T}(\gamma -||\nabla u||^{2})\int_{\Gamma}(h\cdot
\nu
)|\partial _{\nu }u|^{2}\\
=&~2\int_{0}^{T}\int_{\Gamma}\Dn (A^{-1} {u}_{tt})(h\cdot \nabla
u)-2\int_{0}^{T}\int_{\Gamma~}\Dn D_L\left(\Delta
{u}\big|_{\Gamma}\right)(h\cdot \nabla u) \nonumber\\
&+\gamma\int_0^T\int_{\Gamma}(h\cdot
\nu)(\Dn u)^2\cancelto{\le 0}{-\int_{0}^{T}||\nabla u||^2 \int_{\Gamma}(h\cdot
\nu)(\Dn u)^2}\nonumber\\
&-2\int_{0}^{T} \int_{\Gamma} \Dn(A^{-1}
p(\xb))(h\cdot \nabla u).\label{hat1}
\end{align}

\noindent Now, we estimate the integrals on the RHS of the above
inequality, term by term. Using integration by parts in time we
note that
$$\int_{0}^{T}\int_{\Gamma}(\Dn(A^{-1}u_{tt})(h\cdot \nabla u)  = \left(\Dn(A^{-1}u_{t}),\Dn u\right)_{\Gamma}\Big|_{0}^{T}-\int_{0}^{T}\int_{\Gamma}(h\cdot \nu)(\Dn(A^{-1}u_{t})(\Dn u_t).$$

\noindent From the elliptic regularity theorem \cite[Ch. 3]{control}: ~ for any 
$\displaystyle h \in L^2(\Omega)$, we note$~~||A^{-1}h||_2 \le
C||h||_{0,\Omega},$ and thus we have:
\begin{align*} ||\Dn(A^{-1}h)||_{0,\Gamma} \le  ~C||A^{-1}h||_{3/2+\epsilon,\Omega}
\le ~C||A^{-1}h||_{2,\Omega} \le  ~C||h||_{0,\Omega}.
\end{align*}
Hence, compactness of the Sobolev embeddings (and Lions' Lemma \cite[p.108]{kes}),  the H\"{o}lder--Young inequality  yields
that
\begin{equation}\int_{0}^{T}\int_{\Gamma}(\Dn(A^{-1}u_{tt})(h\cdot \nabla u)
\le \epsilon \widehat{E}(T)+
C(\epsilon)||u(t)||_{0,\Omega}+C\widehat{E}(0)+\epsilon
\int_0^T\widehat{E}(\tau)+C(\epsilon)\int_0^T||\Dn
u_t||_{0,\Gamma~}^2 \label{517}.
\end{equation}
Again, from standard elliptic theory, $\Dn D_L \in \mathscr
L\left(L^2(\Gamma), H^{-1}(\Gamma)\right)$. Thus we have
\begin{align}\left|\int_{0}^{T}\int_{\Gamma~}\Dn D_L\left(\Delta  u\big|_{\Gamma}\right)(h\cdot \nabla u)\right| \le& \int_{0}^{T} \left|\left|\Dn D_L\left(\Delta  u\big|_{\Gamma}\right)\right|\right|_{H^{-1}(\Gamma)}\left|\left|h\cdot \nabla u\right|\right|_{H^1(\Gamma)}
\nonumber \\
\le &~C(h) \left(
\left|\left|D(\Dn u_t)\right|\right|_{L^2(0,T
;L^2(\Gamma))}\right)\times \nonumber \\
& \left\{\int_{0}^{T}\left|\left|\dfrac{\partial^2u}{\partial
\nu^2}\right|\right|_{L^2(\Gamma)}^{2}+\left|\left|\dfrac{\partial^2u}{\partial
\tau\partial
\nu}\right|\right|_{L^2(\Gamma)}^{2}\right\}^{1/2}\label{518}
\end{align}

\noindent and similarly,
\begin{equation}
\int_{0}^{T} \int_{\Gamma} \Dn(A^{-1}  p(\xb))(h\cdot \nabla
u)\leq \epsilon \int_0^T \widehat{E}(\tau)+C(p,\epsilon)\cdot
T.\label{519}
\end{equation}

\noindent Taking into account \eqref{517}--\eqref{519} in
\eqref{hat1}, considering $u=0$ on $\Gamma$, employing Assumption
\ref{geom}, using \eqref{pibound}, and applying the H\"{o}lder-Young
  we thus arrive at the next preliminary estimate which can be implemented in \eqref{energy2}:
\begin{align}
-2\int_{0}^{T}&\int_{\Gamma}\Dn(\Delta  u)(h\cdot \nabla
u)-\int_{0}^{T}(\gamma-||\nabla u||^2) \int_{\Gamma}(h\cdot
\nu)(\Dn u)^2\\
\le &~C(h) \left(\left|\left|D(\Dn u_t)\right|\right|_{L^2(0,T
;L^2(\Gamma))}\right)\left\{\int_{0}^{T}\left|\left|\dfrac{\partial^2u}{\partial
\nu^2}\right|\right|_{L^2(\Gamma)}^{2}+\left|\left|\dfrac{\partial^2u}{\partial
\tau\partial
\nu}\right|\right|_{L^2(\Gamma)}^{2}\right\}^{1/2}\nonumber\\
&+C(\epsilon)\int_0^T||\Dn u_t||_{L^2(\Gamma)}^2+\epsilon
\left\{ \widehat{E}(T)+\int_0^T \widehat{E}(\tau)\right\}+C
\widehat{E}(0)+C(p,\epsilon,\gamma,h)\cdot T \nonumber\\
&+\gamma \int_{0}^{T} \int_{\Gamma}(h\cdot
\nu)(\Dn u)^2. \label{tr1}
\end{align}
For the  terms in the last line of \eqref{energy2} involving the
higher order trace term $\partial _{\nu }(h\cdot \nabla u)$, we
note the boundary condition $\Delta u = -D(\Dn u_t)$ on
$\Gamma~$, and we have
\begin{align}
\int_{0}^{T} \int_{\Gamma~}D(\partial _{\nu } u_t)(\partial _{\nu
}u)+\int_{0}^{T}\int_{\Gamma~}2(\Delta u)\partial _{\nu }(h\cdot
\nabla u) \leq&\nonumber \\   ||D(\Dn
u_t)||_{L^2(0,T;\Gamma)}||\Dn
u||_{L^2(0,T;\Gamma)}+||D(\Dn
u_t)&||_{L^2(0,T;L^2(\Gamma))}||\partial_{\nu}(h\cdot \nabla
u)||_{L^2(0,T;L^2(\Gamma))}.\label{T3}
\end{align}

\noindent We now combine \eqref{tr1} and \eqref{T3} in
\eqref{energy2}, and absorb terms to obtain:
\begin{align}
\int_{0}^{T}\widehat{E}(\tau)\leq &~C(h)\left( \left|\left|D(\Dn u_t)\right|\right|_{L^2(0,T
;L^2(\Gamma))}\right)\left\{\int_{0}^{T}\left|\left|\dfrac{\partial^2u}{\partial
\nu^2}\right|\right|_{L^2(\Gamma)}^{2}+\left|\left|\dfrac{\partial^2u}{\partial
\tau\partial
\nu}\right|\right|_{L^2(\Gamma)}^{2}\right\}^{1/2} \nonumber \\
&+C(\epsilon)\int_0^T||\Dn u_t||_{L^2(\Gamma)}^2+\epsilon
\widehat{E}(T)+C
\widehat{E}(0)+C(p,\epsilon,\gamma,h)T \nonumber\\
&+C||D(\Dn u_t)||_{L^2(0,T;\Gamma)}||\Dn
u||_{L^2(0,T;\Gamma)}+C||D(\Dn
u_t)||_{L^2(0,T;L^2(\Gamma))}||\partial_{\nu}(h\cdot \nabla
u)||_{L^2(0,T;L^2(\Gamma))}
\end{align}%

\noindent Now, considering \eqref{energy2} over the interval
$(\alpha,T-\alpha)$ instead of $(0,T)$ (hence, performing the
calculations above on $(\alpha,T-\alpha)$), we can use the
decreasing nature (modulo a constant) of the energy functional
$\big($\eqref{enh} and \eqref{nonen}$\big)$, and then extend the non-critical
integrals back on $(0,T)$. We obtain:
\begin{align}
\int_{\alpha}^{T-\alpha}\widehat{E}(\tau )\leq &~C(h)\left( \left|\left|D(\Dn
u_t)\right|\right|_{L^2(0,T
;L^2(\Gamma))}\right)\left\{\int_{\alpha}^{T-\alpha}\left|\left|\dfrac{\partial^2u}{\partial
\nu^2}\right|\right|_{L^2(\Gamma)}^{2}+\left|\left|\dfrac{\partial^2u}{\partial
\tau\partial
\nu}\right|\right|_{L^2(\Gamma~)}^{2}\right\}^{1/2}\nonumber \\
&+C(\epsilon)\int_0^T||\Dn u_t||_{L^2(\Gamma)}^2+\epsilon
 \widehat{E}(T-\alpha)+C
\widehat{E}(\alpha)+C(p,\epsilon,\gamma,h)\cdot [T-2\alpha] \nonumber\\
&+C||D(\Dn u_t)||_{L^2(0,T;\Gamma)}||\Dn
u||_{L^2(\alpha,T-\alpha;\Gamma)}\nonumber\\&+C||D(\Dn
u_t)||_{L^2(0,T;L^2(\Gamma))}||\partial_{\nu}(h\cdot \nabla
u)||_{L^2(\alpha,T-\alpha;L^2(\Gamma))} \label{energy3}
\end{align}%

The following sharp regularity result for the boundary traces of
solutions to the Euler-Bernoulli equation (linear, with given RHS)
will be critically used:

\begin{theorem}\label{sharpt}(\cite[p.281,Theorem 2.4 ]{sharptrace})
Let $0 < \alpha <T$, $0<\delta<1/2$, and $s_0<1/2$ be arbitrary.
Then the generalized solutions of \eqref{plate} with $(HD)$
boundary conditions enjoy the bound
\begin{align}
\Big[\int_{\alpha}^{T-\alpha} \int_{\Gamma}|\partial_{\tau
\tau}u|^2+|\partial_{\nu \nu}u|^2+|\partial_{\tau \nu}u|^2 d\Gamma
dt \Big]^{1/2} \le&~
~C(T,\alpha,\delta)\Big\{||f(u)||_{L^2(0,T;H^{-s_0}(\Omega))}\notag\\
+||u||_{L^2(0,T;H^{2-\delta}(\Omega))}
+&||D(\partial_{\nu}u_t)||_{L^2(0,T;L^{2}(\Gamma))}+||\partial_{\nu}u_t||_{L^2(0,T;L^{2}(\Gamma))}\Big\}\label{trace}
\end{align}
\end{theorem}

\noindent Hereafter, the explicit dependence of the above
constants on $\alpha,\delta,\gamma,h$ will be suppressed. For the
first term of RHS of above inequality, we make use of
\eqref{trace} and  recall the nonlinear energy $\widehat E$; from  Young's inequality we obtain:
\begin{align}\nonumber
||D(\Dn&
u_t)||_{L^2(0,T;L^2(\Gamma))}\Big\{\int_{\alpha}^{T-\alpha}\left|\left|\dfrac{\partial^2u}{\partial
\nu^2}\right|\right|_{L^2(\Gamma)}^{2}+\left|\left|\dfrac{\partial^2u}{\partial
\tau\partial \nu}\right|\right|_{L^2(\Gamma)}^{2}\Big\}^{1/2}\\
\le&~C +C(T)||D(\Dn
u_t)||^2_{L^2(0,T;L^2(\Gamma))}\Big[\int_0^T\left|\left|(\gamma-||\nabla
u||^2)\Delta u\right|\right|_{-s_0}^2 d\tau \Big]\nonumber\\
&+C(T)\left[||D(\Dn u_t)||_{L^2(0,T;L^2(\Gamma))}^2+||\Dn
u_t||_{L^2(0,T;L^2(\Gamma))}^2
+||u||_{L^2(0,T;H^{2-\delta}(\Omega))}^2\right] \nonumber\\\le
&~C+C(T)||D(\Dn
u_t)||^2_{L^2(0,T;L^2(\Gamma))}\Big[\int_0^T\left(\widehat E(\tau)\right)^2
d\tau \Big] \nonumber\\
&+C(T)\left[||D(\Dn u_t)||_{L^2(0,T;L^2(\Gamma))}^2+||\Dn
u_t||_{L^2(0,T;L^2(\Gamma))}^2
+C||u||_{L^2(0,T;H^{2-\delta}(\Omega))}^2\right]\nonumber\\
\le &~C+\epsilon \int_{\alpha}^{T-\alpha}
\widehat{E}(\tau)+C\widehat{E}(\alpha)+C(T,\epsilon)+C(T)||D(\Dn
u_t)||^2_{L^2(0,T;L^2(\Gamma))}\Big[\int_0^T\left(\widehat E(\tau)\right)^2
d\tau \Big]\nonumber\\
&+C(T)\left[||D(\Dn u_t)||_{L^2(0,T;L^2(\Gamma))}^2+||\Dn
u_t||_{L^2(0,T;L^2(\Gamma))}^2\right].\label{T2}
\end{align}

\noindent We proceed with estimating the last line of the RHS of
\eqref{energy3}. Using Young's inequality, the trace theorem, and \eqref{pibound} we
get
\begin{align*}
||D(\Dn u_t)||_{L^2(0,T;\Gamma)}||\Dn
u||_{L^2(\alpha,T-\alpha;\Gamma)}\leq \epsilon \int_{\alpha}^{T-\alpha}
\widehat{E}(\tau)+C\widehat{E}(\alpha)+\int_0^T||D(\Dn
u_t)||_{L^2(\Gamma)}^2+ C(\epsilon)\cdot T.
\end{align*}

\noindent Considering \begin{equation}
\partial_{\nu}(h\cdot \nabla u) = \partial_{\nu}u+(h \cdot \nu)\partial_{\nu \nu}u+(h\cdot \tau)\partial_{\nu\tau}
u, \label{512}
\end{equation} and applying similar steps as before,
 we have:
\begin{align*}
||D(\Dn u_t)||_{L^2(0,T;L^2(\Gamma))}&||\partial_{\nu}(h\cdot
\nabla u)||_{L^2(\alpha,T-\alpha;L^2(\Gamma))}\nonumber \le
\\~&\epsilon \int_{\alpha}^{T-\alpha}
\widehat{E}(\tau)+C\widehat{E}(\alpha)+C(T) +C(T)||D(\Dn
u_t)||^2_{L^2(0,T;L^2(\Gamma))}\Big[\int_0^T\left(\widehat E(\tau)\right)^2
d\tau \Big]\nonumber \\
&+C(T)\left[||D(\Dn u_t)||_{L^2(0,T;L^2(\Gamma))}^2+||\Dn
u_t||_{L^2(0,T;L^2(\Gamma))}^2\right].
\end{align*}

\noindent If we take into account the last two inequalities we
have:
\begin{align}\nonumber
||D(\Dn u_t)||_{L^2(0,T;\Gamma)}&||\Dn
u||_{L^2(\alpha,T-\alpha;\Gamma)}+||D(\Dn
u_t)||_{L^2(0,T;L^2(\Gamma))}||\partial_{\nu}(h\cdot \nabla
u)||_{L^2(\alpha,T-\alpha;L^2(\Gamma))}\\
\le &~\epsilon \int_{\alpha}^{T-\alpha}
\widehat{E}(\tau)+C\widehat{E}(\alpha)+C(T)+C(T)||D(\Dn
u_t)||^2_{L^2(0,T;L^2(\Gamma))}\Big[\int_0^T(\widehat E(\tau))^2
d\tau \Big]\nonumber \\
&+C(T)\left[||D(\Dn u_t)||_{L^2(0,T;L^2(\Gamma))}^2+||\Dn
u_t||_{L^2(0,T;L^2(\Gamma))}^2\right]\label{T33}.
\end{align}

\noindent Finally, we recall that from the energy identity and \eqref{enh}, we have for $s\le t$:
\begin{align}\label{facts}
\widehat E(t) \le&~ C_1\widehat E(s)+C_2 \\
[t-s] \widehat E(t) \le&~ C_1 \int_s^t \widehat E(\tau)d\tau +C_2[t-s]
\end{align}

\noindent Thus, integrating the energy equality \eqref{enh}  on $[\alpha, T-\alpha]$ and employing \eqref{nonen}, as well as using \eqref{T2}--\eqref{T33} in \eqref{energy3},
we finally have:

\begin{align*}
(T-2\alpha)\widehat E(T)\leq &~C \widehat{E}(0)+ C(T)||D(\Dn
u_t)||^2_{L^2(0,T;L^2(\Gamma))}\Big[\int_0^T(\widehat E(\tau))^2
d\tau \Big]\\
&+C(T)\left[||D(\Dn u_t)||_{L^2(0,T;L^2(\Gamma))}^2+||\Dn
u_t||_{L^2(0,T;L^2(\Gamma))}^2\right]+C(T),
\end{align*}

\noindent which is the desired observability estimate, written in
terms of the nonlinear, positive energy $\widehat E$.
\end{proof}
\begin{remark}\label{duh}
Going back to the equation to obtain trace estimates (and making use of the sharp trace result \eqref{trace} above) is the
key insight into the present problem. This step connects our
observability type estimate to the long-time behavior analysis of
\cite{c-e-l} by utilizing the fact that {\em the nonlinear energy
appears under the time integration against the damping mechanism.}
\end{remark}

\begin{remark}
We emphasize that at multiple times in the argument above
{\em we rely critically on the structure of the Berger
nonlinearity}---particularly in the ability to drop the term
$$-\int_{0}^{T}||\nabla u||^{2}\int_{\Gamma}(h\cdot \nu
)|\partial _{\nu }u|^{2}$$ via the geometric condition and in the
estimate of higher order trace terms (See Remark \ref{duh}). Since
the von Karman nonlinearity does not possess a structure which
accommodates the trace estimates needed above, we cannot expect
a straightforward approach to yield similar results there. Thus, for
the Berger nonlinearity, in addition to the existence of a compact
global attractor, we also have an explicit estimate on the
absorbing set (not available when the attractor is obtained---for
gradient systems---indirectly through the asymptotic smoothness
property).
\end{remark}
\vspace{0.3cm}\noindent \emph{\textbf{Completion of the Proof of
Theorem \ref{abs-ball}}} \vspace{0.25cm} \\Now, let us denote

\begin{equation}
D_0^T = \int_0^T\int_{\Gamma~}
D(\partial_{\nu}u_t)(\partial_{\nu}u_t).
\end{equation}
\noindent Since $D(0)=0$, we have for $s\geq 2$ that
\[
D(s)=\int_{0}^{s}D^{\prime }(\tau )d\tau =s\int_{0}^{1}D^{\prime
}(\tau s)d\tau \geq s\underset{s\geq 2}{\inf
}\int_{0}^{1}D^{\prime }(\tau s)d\tau
\]%
\[ \hskip2.5cm
\geq s\underset{s\geq 2}{\inf }\int_{1/2}^{1}D^{\prime }(\tau
s)d\tau \geq \frac{s}{2}\underset{s\geq 2}{\inf~~}
\underset{~1/2\leq \tau \leq 1}{\inf
}D^{\prime }(\tau s)=\frac{s}{2}\underset{\xi \geq 1}{\inf }%
D^{\prime }(\xi )\geq \frac{s}{2}m.
\]%
It can be shown in an analogous way that for $s\leq -2$ we have $%
D(s)\leq \frac{s}{2}m.$ Additionally, with a given $\epsilon
>0$ we have
for $s\in (\epsilon ,2)$:
\[
D(s)=\int_{0}^{s}D^{\prime }(\tau )d\tau \geq \int_{0}^{\epsilon
}D^{\prime }(\tau )d\tau \geq m_{\epsilon }~\frac{s}{2}
\]%
with a same inequality for  $s\in (-2,-\epsilon ).$ Hence the
above relations give that:
\[
D(s)s\geq \frac{s^{2}}{2}m\text{ \ \ for }\left\vert s\right\vert \geq 2%
\text{ \ \ \ \ \ and \ \ \ \ }D(s)s\geq m_{\epsilon
}\frac{s^{2}}{2}\text{ \ \ for }\left\vert s\right\vert \geq
\epsilon \text{ .\ }
\]%
Similarly, one can show for $s\geq 0$  that
\[
D(s)\leq s\cdot \max \left\{ \underset{0\leq \left\vert \xi \right\vert \leq 1}%
{\sup }D^{\prime }(\xi ),M\right\}.
\]%
Hence we obtain%
\[
\left( D(s)\right) ^{2}\leq D(s)s\max \left\{ \underset{0\leq
\left\vert \xi \right\vert \leq 1}{\sup }D^{\prime }(\xi
),M\right\}.
\]%
Now, considering the above relations we obtain that
\begin{equation}
\left[||D(\Dn u_t)||_{L^2(0,T;L^2(\Gamma~))}^2+||\Dn
u_t||_{L^2(0,T;L^2(\Gamma~))}^2\right] \leq C(T,\Gamma)+\max
\left\{ \frac{2}{m},\underset{0\leq \left\vert \xi \right\vert
\leq 1}{\sup }D^{\prime }(\xi ),M\right\} D_0^T \label{damp2}
\end{equation}
If we take into account \eqref{damp2} in our observability
estimate \eqref{en*} we have

\begin{align*}
(T-2\alpha)\widehat E(\tau)d\tau \le &~C
\widehat{E}(0)+C(T)D_0^T+C(T)D_0^T\int_0^T (\widehat
E(\tau))^2+K(T)
\end{align*}
where $K(T)$ does not depend on $\widehat{E}(0)$. Since
$$\widehat E(0) \le D_0^T+C\widehat E(T)+C(p,\gamma),$$
we have
\begin{align*}
(T-2\alpha-C)\widehat E(T)\le& ~ C(T)D_0^T\left[1+\int_0^T
\left(\widehat E(\tau)\right)^2d\tau\right]+K(T).
\end{align*}

Now, if we rewrite the last inequality in terms of the full
nonlinear enery $\mathscr E(T)$, use the relation \eqref{nonen}
between $\mathscr E$ and $\widehat E$, and employ the notation
$\mathscr E_M(T) \equiv \mathscr E(T)+M$, where $M$ is the
constant coming from \eqref{not1}, we obtain:
\begin{equation*}
(T-2\alpha-C)\mathscr E_M(T) \le C(T)D_0^T\left[1+\int_0^T\mathscr
E_M^2(\tau) d\tau \right]+K(T)
\end{equation*}
Using the fact that $\mathscr E_M(T) \le \mathscr E_M(0)$, and
$D_0^T =\mathscr E_M(0)-\mathscr E_M(T)$, we have
\begin{equation*}
(T-2\alpha-C)\mathscr E_M(T) \le C(T)\left[\mathscr
E_M(0)-\mathscr E_M(T)\right]\left[1+T\mathscr E^2_M(0)
\right]+K(T),
\end{equation*}
from which we obtain
\begin{equation*}
\Big[(T-2\alpha-C)+C(T)\left[1+T\mathscr E^2_M(0)
\right]\Big]\mathscr E_M(T) \le C(T)\left[1+T\mathscr E^2_M(0)
\right]\mathscr E_M(0)+K(T).
\end{equation*}
Then, assuming $T$ large enough
and recalling \eqref{nonen}, we have
\begin{align*}
\mathscr E_M(T) \le \frac{C(\widehat E(0),T)}{T+C(\widehat
E(0),T)}\mathscr E_M(0)+\frac{K(T)}{T+C(\widehat E(0),T)}
\end{align*}
Denoting ~~$\eta(\widehat E(0),T)=\frac{C(\widehat
E(0),T)}{T+C(\widehat E(0),T)}<1$, we rewrite the last inequality
as
\begin{align*}
\mathscr E_M(T) \le \eta(\widehat E(0),T)\mathscr
E_M(0)+\frac{K(T)}{T+C(\widehat E(0),T)}.
\end{align*}
Now, we can reiterate the same estimate on each subinterval
$(mT,(m+1)T)$ via the semigroup property. We note that the
constants $C(\widehat E(0),T)$ and $K(T)$ will be same at each
step. Then we obtain
\begin{align}
\mathscr E_M((m+1)T) \le & \frac{C(\widehat E(0),T)}{T+C(\widehat
E(0),T)}\mathscr E_M(mT)+\frac{K(T)}{T+C(\widehat
E(0),T)}\nonumber\\
&\leq \eta(\widehat E(0),T)^m \mathscr
E_M(0)+\sum_{i=0}^{m}\eta(\widehat E(0),T)^i
\frac{K(T)}{T+C(\widehat E(0),T)}\nonumber \\
&\leq \eta(\widehat E(0),T)^m \mathscr
E_M(0)+\Big[\frac{1}{1-\eta(\widehat E(0),T)}\Big]\frac{K(T)}{T+C(\widehat E(0),T)}\nonumber \\
&\equiv \eta(\widehat E(0),T)^m \mathscr
E(0)+\overline{K}(T)\nonumber
\end{align}
The monotonicity and continuity of the energy function, and the
fact that $\eta(\widehat E(0),T)<1$, gives
\begin{equation*}
\mathscr E_M(t) \le \overline{K}(T)+1,~~~~~~for~~ all
~~~~t>t_0(\widehat E(0)),
\end{equation*}
which together with \eqref{nonen} completes the proof of Theorem
\ref{abs-ball}.

\subsubsection{Asymptotic Smoothness---Completion of the Proof of Theorem \protect\ref{gatt}}\label{smoothnesssec}
In this section, as the second step of the proof of Theorem
\ref{gatt}, we show the asymptotic smoothness property of the
dynamical system generated by (\ref{plate}) with (HD) boundary
conditions in the following theorem:

\begin{theorem}\label{asympt}
Let Assumptions \ref{ass2} and \ref{geom} be in force. The
dynamical system $\left(S_H(\cdot), \mathscr {H}\right)$ generated
by generalized solutions to \eqref{plate} taken with (HD)
boundary conditions is {\em asymptotically smooth}.
\end{theorem}
We note that after Theorem \ref{asympt} is proved, our main
theorem on the existence of a global attractor for generalized
solutions to \eqref{plate} taken with (HD) conditions (Theorem
\ref{gatt}) will be proved, via Theorem \ref{0}.
\begin{proof}[Proof of Theorem \ref{asympt}]
The proof rests on the application of Theorem \ref{psi}. Hence we
are interested in the difference of two solutions $z=u-w$, where $%
U(t)=(u(t),u_{t}(t))=S_{H}(t)y_{1}$ and
$W(t)=(w(t),w_{t}(t))=S_{H}(t)y_{2}$ solve (\ref{plate})
corresponding to initial conditions $y_{1}$ and $y_{2}$
(respectively), taken from an invariant, bounded set. Then, $z$
will solve the following problem:

\begin{align}
z_{tt}+\Delta^2 z +\mathcal{F}(z)= 0~~ \text{ in } (0,T)\times
\Omega; \label{diffplate} & ~~~~~z=0,~~\Delta z=-[D(\partial _{\nu
}u_{t})-D(\partial _{\nu }w_{t})]~\mathrm{on}~~\Gamma,
\\z(0)=u_0-w_0,~~z_t(0)=u_1-w_1
, \notag
\end{align}
where $ \mathcal F(z)\equiv f_B(u)-f_B(w)$. We denote the energy
for the difference: ~ $\displaystyle
E_{z}(t)=\frac{1}{2}\big\{||\Delta z||^2+||z_{t}(t)||^{2}\big\}.$

\noindent The energy identity \eqref{enh}, the relation
\eqref{nonen} between the energies, and the bound on lower
frequencies \eqref{pibound}, give that there exists an $R_*$ such
that the set
$$\mathscr W_R \equiv \{(u_0,u_1) \in \mathscr H~:~\mathscr E(u_0,u_1) \le R\}$$ is a non-empty bounded set in $\mathscr H$ for all $R \ge R_*$. Moreover, any bounded set $\mathscr B \subset \mathscr H$ is contained in $\mathscr W_R$ for some $R$, and the set $\mathscr W_R$ is invariant with respect to $S_H(t)$.
Then, we consider the restriction of the dynamical
system $(S_H(\cdot), \mathscr {H})$ to $(S_H(\cdot), \mathscr
W_R)$ in showing the asymptotic smoothness property, and thus we
consider the solutions $u,w$ satisfying
\begin{equation*}
||u(t)||_{2}+||u_{t}(t)||_{0}+||w(t)||_{2}+||w_{t}(t)||_{0}\leq
C(R),~~t>0.
\end{equation*}
The main ingredient will be the following estimate on the
difference of the solutions:

\begin{lemma}
Let $T>0$ and $\beta \in C^2(\mathbb R)$ be a given function
satisfying (i) $\text{supp}(\beta) \subset [\alpha, T-\alpha]$
(with $\alpha < T/2$); (ii) $0\le \beta \le 1$ and $\beta \equiv
1$ on $[\alpha, T-\alpha]$. Then, any solution $z$ to equation
\eqref{diffplate} satisfies
\begin{align}
\int_0^T E_z(t)\beta(t)dt\leq &~ C_1 \int_0^T
E_z(t)|\beta'(t)|dt+C_2(T)\left\{\int_0^T \int_{\Gamma~}
\Big[D(\partial _{\nu }u_{t})-D(\partial _{\nu
}w_{t})\Big]\partial _{\nu }z_t+|\partial _{\nu }z_t|^2d\Gamma\right\}\nonumber\\
+& l.o.t.(z),\label{e1}
\end{align}
where $$l.o.t.(z)\equiv
C(T,R)\Big[\sup_{[0,T]}||z(t)||_{2-\eta}^{2} +
\int_0^T||z_t(\tau)||_{-\eta}^2 d\tau\Big],~~0<\eta <1/2.$$
\end{lemma}

\begin{proof}
The proof of this lemma follows \cite[Section 10.5]{springer}
exactly, with respect to the linear portion of the dynamics and the
nonlinear damping: the application of the flux multiplier $h\cdot
\nabla (\beta z)$, where $h=\xb-\xb_0$, $\xb_0\in \mathbb{R}^2$,
the use of sharp trace results proved in \cite{sharptrace}
(as in
 the previous section), the local Lipschitz
property of $f$, and the employment of Assumption \ref{ass2}. This yields the above estimate \eqref{e1}.
\end{proof}

\noindent Now, by the energy relation,
\begin{equation}
E_z(T)+\int_t^T\int_{\Gamma~} \left[D(\partial_{\nu}
u_t)-D(\partial_{\nu}w_t)\right](\partial_{\nu}z_t)+\int_t^T(\mathcal
F(z),z_t)=E_z(t),~~T\ge t,\label{enz}
\end{equation}
we have:
\begin{enumerate}\item[(i)] \begin{align*}\int_0^T E_z(t)(1-\beta(t)+|\beta'(t)|)dt\le &\\
\int_0^T\big[1-\beta(t)+|\beta'(t)|\big]&dt\Big[E_z(T)+\int_0^T\int_{\Gamma~}
\left[D(\partial_{\nu}
u_t)-D(\partial_{\nu}w_t)\right](\partial_{\nu}z_t)\Big]+\mathcal
F_*(z),\end{align*} where
$$\mathcal{F}_*(z) \equiv \int_0^T(1-\beta(t)-|\beta'(t)|)\int_t^T(\mathcal F(z),z_t) d\tau dt$$

\item[(ii)] $$TE_z(T)\leq \int_0^T E_z(t)dt +\mathcal F_{**}(z),$$where
$$\mathcal F_{**}(z)\equiv -\int_0^T\int_t^T(\mathcal F(z),z_t)
d\tau dt.$$ \end{enumerate}

\noindent Combining (i) and (ii), we obtain that there exist
$T_0>0$, and constants $C_1(T)$ and $C_2(R,T)$, such that
\begin{align}
TE_{z}(T)+\int_0^TE_z(t) dt \leq &~
C_1(T)\left[\int_0^T\int_{\Gamma~} \left[D(\partial_{\nu}
u_t)-D(\partial_{\nu}w_t)\right](\partial_{\nu}z_t)d\Gamma~dt+
\int_0^T\int_{\Gamma~}\left|\partial_{\nu}z_t\right|^2  d\Gamma~dt\right] \nonumber\\
&+\mathcal F_*(z)+\mathcal F_{**}(z)+C_2(R,T)l.o.t.(z),~~\forall ~~T\geq T_0.\label{diff-est1} \end{align}%
 Using the assumption on the structure of the damping (Assumption \ref{ass2}), we have
$$\int_0^T\int_{\Gamma~} \left| \partial_{\nu}z_t \right|^2 \le \epsilon +C(\epsilon)
\int_0^T\int_{\Gamma~} \left[D(\partial_{\nu}
u_t)-D(\partial_{\nu}w_t)\right](\partial_{\nu}z_t),
~~\forall~~\epsilon>0.$$ On the other hand, it follows immediately
from the energy relation \eqref{enz} that
$$\int_0^T \int_{\Gamma~} \left[D(\partial_{\nu} u_t)-D(\partial_{\nu}w_t)\right](\partial_{\nu}z_t)
\le E_z(0)-E_z(T) +\left|\int_0^T(\mathcal F(z),z_t)\right|.$$
Now, taking into account the last two inequalities in
\eqref{diff-est1}, we obtain
\begin{align}
TE_{z}(T)+\int_0^TE_z(t)dt \leq &~\epsilon+ C_1(\epsilon,T)\left[E_z(0)-E_z(T) \right]+C_2(\epsilon,T)\left|\int_0^T(\mathcal F(z),z_t)\right| \nonumber\\
&+\mathcal F_*(z)+\mathcal
F_{**}(z)+C_3(\epsilon,T,R)l.o.t.(z)\label{diff-est} \end{align}
for any $\epsilon>0$.

In order to apply Theorem \ref{psi}, we need to simplify the
nonlinear terms above. For this, we demonstrate a decomposition of
Berger nonlinearity for the difference of two solutions.
\begin{lemma}\label{decomp}
Let $z=u-w$, and let $f_B(u)=(\gamma -||\nabla u||^2)\Delta u$,
and $\mathcal F(z) = f_B(u)-f_B(w)$. Also assume $u,w \in
C([t,T];(H^2\cap H_0^1)(\Omega))\cap C^1([t,T];L^2(\Omega))$. Then
\begin{align*}
\int_t^T (\mathcal F(z),z_t )_{\Omega} d\tau  =&\Big[( \mathcal
F(z),z)_{\Omega} +\dfrac{\gamma}{2}||\nabla
z||^2-\dfrac{1}{2}||\nabla u||^2||\nabla z||^2 -(||\nabla
u||^2-||\nabla w||^2)( \Delta w,z )_{\Omega}
\Big]_t^T\\&-\int_t^T(||\nabla u||^2-||\nabla w||^2)( \Delta w,z_t
)_{\Omega} d\tau +\int_t^T(\Delta u,u_t )_{\Omega}||\nabla z||^2
d\tau
\end{align*}
\end{lemma}
\begin{proof}
Letting $z=u-w$, we note
\begin{align*}
\int_t^T ( \mathcal F(z),z_t ) d\tau =&~-\int_t^T ( [||\nabla u||^2\Delta u-||\nabla w||^2\Delta w],z_t ) d\tau +\gamma \int_t^T ( \Delta z,z_t ) d\tau \\
=& \int_t^T ( [||\nabla u||^2\Delta u]'-[||\nabla w||^2\Delta
w]',z) d\tau  +\left[( \mathcal F(z),z) +\dfrac{\gamma}{2}||\nabla
z||^2\right]_t^T,\nonumber
\end{align*}
where we have integrated by parts in time, integrated by parts in
space on the term with $\gamma$, and recognized the total
derivative term $( \nabla z, \nabla z_t )$. We now focus on the
nonlinear term and observe
\begin{align*}
\int_t^T ( [||\nabla u||^2\Delta u]'-[||\nabla w||^2\Delta w]',z)
d\tau =& \int_t^T( (||\nabla u||^2)'\Delta u - (||\nabla
w||^2)'\Delta w,z ) d\tau \\ \nonumber &+\int_t^T( ||\nabla
u||^2\Delta u_t-||\nabla w||^2\Delta w',z) d\tau.
\end{align*}
If we add and subtract the terms which are mixed in $u$ and $w$ we
have
\begin{align*}
\int_t^T \big( [||\nabla u||^2\Delta u]'&-[||\nabla w||^2\Delta w]',z\big) d\tau =\\
& \int_t^T\Big[(||\nabla u||^2)'( \Delta z,z) +\big[(||\nabla
u||^2)'-(||\nabla w||^2)' \big]( \Delta w,z) \Big]d\tau\nonumber
\\\nonumber &+\int_t^T\Big[||\nabla u||^2( \Delta z_t,z
)+[||\nabla u||^2-||\nabla w||^2]( \Delta w_t,z) \Big]d\tau.
\end{align*}

Now, recognizing a total derivative and adjusting we arrive at
\begin{align*}
\int_t^T \big( [||\nabla u||^2\Delta u]'&-[||\nabla w||^2\Delta
w]',z\big) d\tau  \\ =&-
\int_t^T\Big[\dfrac{d}{d\tau}\big[||\nabla u||^2||\nabla
z||^2\big]-\dfrac{1}{2}||\nabla u||^2(||\nabla
z||^2)'\Big]d\tau\\& -\int_t^T\Big[\dfrac{d}{d\tau}\big[(||\nabla
u||^2-||\nabla w||^2)( \Delta w, z )\big]+(||\nabla u||^2-||\nabla
w||^2)( \Delta
w, z_t) \Big]d\tau\\
=& -\Big[\dfrac{1}{2}||\nabla u||^2||\nabla z||^2 +(||\nabla
u||^2-||\nabla w||^2)( \Delta w,z ) \Big]_t^T\\
& -\int_t^T(||\nabla u||^2-||\nabla w||^2)( \Delta w,z_t ) d\tau
+\int_t^T(\Delta u,u_t )||\nabla z||^2 d\tau\nonumber
\end{align*}
which is the desired decomposition.
\end{proof}
\noindent At this point, noting that \begin{equation*}
||u(t)||_{2}+||u_{t}(t)||_{0}+||w(t)||_{2}+||w_{t}(t)||_{0}\leq
C(R),~~t>0,
\end{equation*}
using triangle inequality
\begin{align*}\left|~||\nabla u||^2-||\nabla w||^2~\right|=&~\Big|||\nabla u||-||\nabla w||\Big|\left(||\nabla u||+||\nabla w||\right)\\
\le &~||\nabla u - \nabla w||\left(||\nabla u||+||\nabla w||\right)\\
\le &~C(R)||z||_1,
\end{align*}
and taking into account the last two inequalities in Lemma
\ref{decomp} we obtain the {\em key inequality}:
\begin{equation}\label{usethis}\left|\int_t^T(\mathcal F(z),z_t)d\tau \right| \le \epsilon \int_t^TE_z(\tau)d\tau+C(\epsilon, R,T)\sup_{[t,T]} ||z||^2_{2-\eta},~~\eta>0.\end{equation}
From this, the analogous bounds on $\mathcal F_*$ and $\mathcal
F_{**}$ follow. \vspace{0.25cm}

\noindent\textbf{\em{Completion of the Proof of Theorem
\ref{asympt}--Asymptotic Smoothness}} \vspace{0.25cm}

\noindent We are now in a position to finish the proof of Theorem
\ref{asympt}. Firstly, implementing \eqref{usethis} in
\eqref{diff-est} and rescaling $\epsilon$, we see that
\begin{equation} E_z(T) \le \epsilon
+C_1(\epsilon,T)[E_z(0)-E_z(T)]+l.o.t.(z).
\end{equation}

\noindent In order to invoke Theorem \ref{psi} we need only to
construct a functional $\Psi _{\epsilon ,R,T}$ satisfying the
compensated compactness condition \cite{springer,kh}, i.e.,
\begin{equation*}
\underset{k\rightarrow \infty }{\lim \inf
}~~\underset{n\rightarrow \infty }{ \lim \inf }~\Psi _{\epsilon
,R,T}(y_{n},y_{k})=0
\end{equation*}
for every sequence $\{y_{n}\}$ from $\mathscr B \subset \mathscr
H$. For this, we use a standard stabilization type argument via
the semigroup property of the dynamics on the interval
$((m-1)T,mT)$. This yields
$$E_z(mT) \le \dfrac{\epsilon}{1+C(\epsilon,R,T)}+\gamma_{\epsilon}E_z((m-1)T)+l.o.t.^m(z),$$ where
$$\gamma_{\epsilon} = \dfrac{C(\epsilon, R,T)}{1+C(\epsilon,R,T)}<1$$ and
$$l.o.t.^m(z) = C(\epsilon,R,T)\Big[\sup_{[(m-1)T,mT]} ||z(\tau)||_{2-\eta}^2 +\int_{(m-1)T}^{mT}||z_t(\tau)||_{-\eta}^2 d\tau\Big].$$
After iteration, we observe that
$$E_z(mT) \le \epsilon+\gamma_{\epsilon}^mE_z(0)+\sum_{j=0}^{m-1}\gamma_{\epsilon}^j l.o.t.^{m-j}(z).$$
For any $\epsilon_0$, we may select $\epsilon>0$, then $m$ large
enough so that, with $R$ and $T$ fixed, we have
$$E_z(mT) \le \epsilon_0+\sum_{j=0}^{m-1}\gamma_{\epsilon}^j l.o.t.^{m-j}(z).$$
Now, for $y_1=U(0)$ and $y_2=W(0)$, setting
\begin{equation*} \Psi_{\epsilon_0 ,\mathscr{B},mT}(y_1, y_2) \equiv
\sum_{j=0}^{m-1}\gamma_{\epsilon}^j l.o.t.^{m-j}(z)
\end{equation*}
the functional $\Psi_{\epsilon_0 ,\mathscr B,mT}$---being compact
with respect to the energy space $\mathscr H$ (via the Sobolev
embeddings)---satisfies the requisite condition of the functional
$\Psi$ in Theorem \ref{psi}. This concludes the proof
of asymptotic smoothness of the dynamical system
$(S_H(\cdot),\mathscr W_R)$ for any $R>0$, according to Theorem
\ref{psi}.
\end{proof}
As a result, by Theorem \ref{abs-ball} and Theorem \ref{asympt} we
obtain the existence of global attractor which finishes the proof
of Theorem \ref{gatt}.

\section{Acknowledgements}
The authors would like to thank Professors Azer Khanmamedov and
Irena Lasiecka for their (respective) very helpful advice and
suggestions. The authors also thank the referee for valuable suggestions which improved this treatment.

J.T. Webster was partially supported by NSF-DMS-1504697 in
performing this research.

\section{Appendix}

\subsection{Long-time Behavior of Dynamical Systems}

In the context of this paper we will use a few keys theorems
(which we now formally state) to prove the existence of the
attractor and determine its properties, as well as provide some
context for other results mentioned in the discussions above. For
general dynamical systems references, see
\cite{babin1,miranville,raugel,springer} (and refences therein).
For proofs pertinent to what is presented here, and more
references, see \cite{springer}.

 Let $(%
\mathcal{H},S_t)$ be a dynamical system on a complete metric space $\cH$  with $\mathscr{N}\equiv \{x\in%
\mathcal{H}:S_tx=x~\text{ for all }~ t \ge 0\}$ the set of its
stationary points. $(\cH, S_t) $ is said to be dissipative iff it
possesses a bounded absorbing set $\mathcal B$. This is to say
that for any bounded set $B$, there is a time $t_B$ so that
$S_{t_B}(B)\subset \mathcal B$. We say that a dynamical system is
\textit{asymptotically compact} if there exists a compact set $K$
which is uniformly attracting: for any bounded set $ D\subset
\mathcal{H}$ we have that $\displaystyle~
\lim_{t\to+\infty}d_{\mathcal{H}}\{S_t D|K\}=0$~ in the sense of
the Hausdorff semidistance. $(\mathcal{H},S_t)$ is said to be
\textit{asymptotically smooth} if for any bounded, forward
invariant $(t>0) $ set $D$ there exists a compact set $K \subset
\overline{D}$ which is uniformly attracting (as in the previous definition).
\textit{Global attractor} $\mathbf{A}$ is a closed, bounded set in $%
\mathcal{H}$ which is invariant (i.e. $S_t\mathbf{A}=\mathbf{A}$ for all $%
t>0 $) and uniformly attracting.

The following   {\it if and only if}  characterization of global
attractors  is well-known \cite{babin1,springer}
\begin{theorem}\label{0}
Let $(\cH, S_t)$ be a dissipative dynamical system in a complete
metric space  $\cH$. Then $(\cH, S_t ) $ possesses a compact
global attractor $\mathbf{A} $ if and only if $(\cH, S_t) $ is
asymptotically smooth.
\end{theorem}

 An {\it asymptotically smooth}
dynamical system for which there is a Lyapunov function  $\Phi(x)
$  that is  bounded from above  on any bounded set
 can be thought of as one which possesses \textit{local
attractors}. To see this stated precisely see \cite{ch-l} page 33.
Such a result provides an existence of {\it local attractors},
i.e., and attractor for any bounded set of initial data. However,
these sets need not be uniformly bounded  with respect to the size
of the set of initial data.
 The latter can be guaranteed by the existence of a uniform absorbing  set.
However, establishing this existence of an absorbing set may be
technically demanding. In some instances, there is a way of
circumventing this difficulty which takes advantage of the ``good"
structure of a Lyapunov function.

A \textit{strict Lyapunov function} for $(\mathcal{H},S_t)$ is a functional $%
\Phi$ on $\mathcal{H}$ such that (i) the map $t \to \Phi(S_tx)$ is
nonincreasing for all $x \in \mathcal{H}$, and (ii)
$\Phi(S_tx)=\Phi(x)$ for
all $t>0$ and $x \in \mathcal{H}$ implies that $x$ is a stationary point of $%
(\mathcal{H},S_t)$. If the dynamical system has a strict Lyapunov
function defined on the entire phase space, then we say that
$(\mathcal{H},S_t)$ is \textit{gradient}.

We can address attractors for gradient systems and characterize
the attracting set. The following result follows from Theorem 2.28
and Corollary 2.29 in \cite{ch-l}.

\begin{theorem}
\label{gradsmooth} Suppose that $(\mathcal{H},S_t)$ is a gradient,
asymptotically smooth dynamical system. Suppose its Lyapunov
function $\Phi (x)$ is bounded from above on any bounded subset of
$\mathcal{H}$ and the set $\Phi _{R}\equiv \{x\in \mathcal{H}:\Phi
(x)\leq R\}$ is bounded for every $R$. If the set of stationary
points for $(\mathcal{H},S_t)$ is
bounded, then $(\mathcal{H},S_t)$ possesses a compact global attractor $%
\mathbf{A}$ which coincides with the unstable manifold, i.e.
\begin{equation*}
\mathbf{A}=\mathscr{M}^{u}(\mathscr{N})\equiv \{x\in
\mathcal{H}:~\exists
~U(t)\in \mathcal{H},~\forall ~t\in \mathbb{R}~\text{ such that }~U(0)=x~%
\text{ and }~\lim_{t\rightarrow -\infty }d_{\mathcal{H}}(U(t)|\mathscr{N}%
)=0\}.
\end{equation*}
\end{theorem}

Secondly, we state a useful criterion (first appearing \cite{kh}
and stated in the present version in \cite{springer}) which
reduces asymptotic smoothness to finding a suitable functional on
the state space with a compensated compactness condition:

\begin{theorem}(\cite{ch-l}-Proposition 2.10)
\label{psi} Let $(\mathcal{H},S_t)$ be a dynamical system,
$\mathcal{H}$ Banach with norm $||\cdot||$. Assume that for any
bounded positively invariant set $B \subset \mathcal{H}$ and for
all $\epsilon>0$ there exists a $T\equiv T_{\epsilon,B}$ such that
\begin{equation*}
||S_Tx_1 - S_Tx_2||_{\mathcal{H}} \le
\epsilon+\Psi_{\epsilon,B,T}(x_1,x_2),~~x_i \in B
\end{equation*}
with $\Psi$ a functional defined on $B \times B$ depending on
$\epsilon, T,$ and $B$ such that
\begin{equation*}
\liminf_m \liminf_n \Psi_{\epsilon,T,B}(x_m,x_n) = 0
\end{equation*}
for every sequence $\{x_n\}\subset B$. Then $(\mathcal{H},S_t)$ is
an asymptotically smooth dynamical system.
\end{theorem}

\end{document}